\documentclass[12pt]{amsart}
\usepackage{amsthm}
\usepackage{amssymb}
\usepackage{amsmath}
\usepackage{amsfonts}
\usepackage{amsthm}
\usepackage{bm}
\usepackage{color}
\usepackage{graphicx}
\usepackage[all]{xy}
\usepackage{mathtools}
\usepackage{euscript}
\usepackage{mathrsfs}
\usepackage{soul}

\newcommand{\ideal}[1]{{\left\langle#1\right\rangle}}
\newcommand{\excise}[1]{}

\newcommand{\p}{\partial}
\newcommand{\nn}{{\bf 1}}

\newtheorem{theorem}{Theorem}[section]
\newtheorem{lemma}[theorem]{Lemma}

\newtheorem{corollary}[theorem]{Corollary}

\theoremstyle{definition}

\newtheorem{remark}[theorem]{Remark}

\newtheorem{definition}[theorem]{Definition}

\newtheorem{assumption}[theorem]{Assumption}

\newcommand{\baseRing}[1]{\ensuremath{\mathbb{#1}}}
\newcommand{\Z}{\baseRing{Z}}
\newcommand{\C}{\baseRing{C}}
\newcommand{\N}{\baseRing{N}}
\newcommand{\R}{\baseRing{R}}

\newcommand{\too}{{\longrightarrow}}

\def\<{\langle}
\def\>{\rangle}
\def\0{\mathbf{0}}

\def\CC{{\mathbb C}}

\newcommand{\NN}{{\mathbb N}}
\def\PP{{\mathbb P}}
\def\QQ{{\mathbb Q}}
\def\RR{{\mathbb R}}

\def\ZZ{{\mathbb Z}}

\def\cA{{\mathcal A}}

\def\cC{{\mathcal C}}
\def\cD{{\mathcal D}}
\def\cE{{\mathcal E}}

\def\cO{{\mathcal O}}
\def\cP{{\mathcal P}}

\def\cS{{\mathcal S}}

\def\cU{{\mathcal U}}

\def\rank{{\rm rank\ }}
\def\supp{\text{\rm supp}}

\def\ol#1{{\overline {#1}}}
\def\wt#1{{\widetilde {#1}}}

\numberwithin{equation}{section}
\newcommand\restr[2]{{
  \left.\kern-\nulldelimiterspace
  #1
  \vphantom{\big|}
  \right|_{#2}
  }}

\date{}
\begin{document}

\title{Gevrey expansions of hypergeometric integrals II}\thanks{The three authors are partially supported by MTM2016-75024-P and FEDER. The first two authors are also partially supported by FQM333-Junta de Andaluc\'{\i}a}

\author[F.-J. Castro-Jim\'enez]{Francisco-Jes\'us Castro-Jim\'enez}
\address{Departamento de \'Algebra e Instituto de Matemáticas-IMUS, Universidad de Sevilla, Av. Reina Mercedes s/n 41012 Sevilla, Spain.} \email{castro@us.es}
\thanks{}

\author[M.-C. Fern\'andez-Fern\'andez]{Mar\'{\i}a-Cruz Fern\'andez-Fern\'andez}
\address{Departamento de \'Algebra, Universidad de Sevilla, Av. Reina Mercedes s/n
41012 Sevilla, Spain.} \email{mcferfer@us.es}
\thanks{}

\author[M. Granger]{Michel  Granger}
\address{Universit\'e d'Angers, D\'epartement de Math\'ematiques, LAREMA, CNRS
UMR n. 6093, 2 Bd. Lavoisier, 49045 Angers, France.}
\email{michel.granger@univ-angers.fr}
\thanks{}

\begin{abstract}
We study integral representations of the Gevrey series solutions of irregular hypergeometric
systems under certain assumptions. We prove that, for such systems, any Gevrey series solution, along a coordinate hyperplane of its singular support, is the asymptotic expansion of a holomorphic solution given by a carefully chosen integral representation.
\end{abstract}

\maketitle

\today

\section{Introduction.}\label{intro-0}
In \cite{GGZ87} the authors introduced the notion of general $A$-hypergeometric system generalizing Gauss hypergeometric equation and other classical differential equations.  In {\it loc. cit.} and in a series of papers (see \cite{GZK89}, \cite{GKZ90} and the references therein) the authors analyzed the solutions of such systems developing the theory of generalized hypergeometric functions.

General $A$-hypergeometric systems, also known as GKZ systems, are finitely generated $D$-modules where $D:=\CC[x]\langle \partial \rangle = \CC[x_1,\ldots,x_n]\langle \partial_1,\ldots,\partial_n\rangle$ stands for the complex $n$-th Weyl algebra.

The input data for a GKZ system is a pair $(A,\beta)$ where $\beta$ is a vector in $\CC^d$ and $A=(a_{k\ell})=(a{(1)},\ldots , a{(n)})\in (\Z^{d})^n$ is a $d\times n$ matrix whose $\ell^{\rm{th}}$ column is $a{(\ell)}$ and $\ZZ A:=\sum_{k=1}^d \ZZ a(k) =\ZZ^d$. The {\em toric ideal} $I_A \subset \C[\partial]:=\C[\partial_1,\ldots,\partial_n]$ is the ideal generated by the family of binomials $\partial^u -\partial^v$ where $u,v\in \N^n$ and $Au=Av$ (we assume $0\in \NN$). The polynomial ideal $I_A$ is prime and the Krull dimension of the quotient ring $\CC[\partial]/I_A$ equals $d$. Following \cite{GGZ87, GZK89}, the hypergeometric ideal associated with the pair $(A,\beta)$ is~: $$H_A(\beta):=DI_A + D(E_1-\beta_1,\ldots, E_d-\beta_d)$$ where $E_k= \sum_{\ell=1}^n a_{k\ell}x_\ell\partial_\ell$ is the $k^{\rm{th}}$ Euler operator associated with the $k^{\rm{th}}$ row of $A$. The
corresponding hypergeometric $D$--module (or $A$-hypergeometric system) is the
quotient left $D$--module $M_A(\beta):= \frac{D}{H_A(\beta)}$.

Hypergeometric systems $M_A(\beta)$ are holonomic $D$--modules on $X=\CC^n$, \cite{GZK89} and \cite[Thm.~3.9]{Adolphson}. Moreover, a characterization of the regularity of $M_A(\beta)$, in the sense of $D$--module theory \cite{Mebkhout-positivite}, \cite{Laurent-Mebkhout},  is provided in the series of papers \cite{Hotta}, \cite{SST} and \cite{schulze-walther}:
The holonomic $D$-module  $M_A(\beta)$ is {\em regular}  if and only if the toric ideal $I_A$ is homogeneous for the standard grading
in the polynomial ring $\CC[\partial]$. In particular the condition to be regular for $M_A(\beta)$ is independent of the parameter vector $\beta$.

The dimension of the space of germs of holomorphic solutions of $M_A(\beta)$ around a generic point in $X$ equals $d! {\rm Vol}(\Delta_A)$ if $\beta$ is generic (see \cite{GZK89}, \cite[Cor.~5.20]{Adolphson} and \cite{MMW}). Here $\Delta_A$ is the convex hull in $\R^d$ of the points ${\bf 0},a(1),\ldots, a(n)$, where ${\bf 0}\in \RR^d$ is the origin, and ${\rm Vol}(\Delta_A)$ is its Euclidean volume. These holomorphic solutions are represented as $\Gamma$--series in \cite{GZK89} (see also \cite{Ohara-Takayama} and \cite{F}) when $\beta$ is generic enough.

A. Adolphson considers in \cite[Sec.~2]{Adolphson} integral representations of solutions of $M_A(\beta)$ which involve exponentials of polynomial functions and appropriate integration cycles. In \cite{Esterov-Takeuchi-2015} A. Esterov and K. Takeuchi prove that the generic holomorphic solution spaces are in fact completely described by Adolphson's integral representations along rapid decay cycles as introduced by M. Hien in \cite{Hien07} and \cite{Hien09}. Such type of integrals are also used in \cite{M-H18b} and generalized in \cite{M-H19}, where they are called Laplace integrals.

The slopes, see \cite{Laurent-Mebkhout}, of $M_A(\beta)$ along coordinate subspaces are described in \cite{schulze-walther} and their corresponding irregularity sheaves and Gevrey series solutions, see \cite{Mebkhout-positivite}, are studied and described for generic parameters $\beta$ in \cite{F} (see also \cite{FC2,FC1}). Moreover, in \cite[Prop. 5.3 and Rmk. 5.4]{CFKT} these Gevrey series solutions of $M_A(\beta)$ are interpreted as asymptotic expansions of certain of its holomorphic solutions under some assumption on the Gevrey index of the series, via the so-called {\em modified A-hypergeometric systems} introduced in \cite{T09}.

In \cite{Castro-Granger-IMRN}, and when $A$ is a row matrix with positive integer entries,  the authors develop a link between Gevrey series solutions of $M_A(\beta)$ and holomorphic solutions in sectors following Adolphson's approach. They prove that any Gevrey series
solution, along the singular support of the system $M_A(\beta)$, is the asymptotic expansion of a holomorphic
solution given by a carefully chosen integral representation.

In this paper we further develop this link when the matrix $A=(a(1),\ldots,a(n))\in (\ZZ^{d})^n$ satisfies two conditions. Since the rank of $A$ is assumed to be $d$, we may also assume,  after a possible reordering of the columns, that the first $d$ columns of $A$ determine a $(d-1)$--simplex $\sigma$. We further assume that $A$ satisfies the following two conditions (see Assumption \ref{assumption}): (1) The points $a(d+1), \ldots, a(n-1)$ belong to the interior of the convex hull $\Delta_\sigma$ of $\sigma$ and the origin; and (2) The point $a(n)$ is not in $\Delta_\sigma$ and belongs to the open positive cone of $\sigma$. Figure 1 shows an example of an allowed column set configuration for a $2\times5$ matrix $A$, where $\Delta_{\sigma}$ is the triangle.

\setlength{\unitlength}{11mm}
$$\begin{picture}(-6,5)(10,3)
\put(5,4){\vector(0,1){4}}
\put(4,5){\vector(1,0){5}}
\put(7,5){\line(-1,2){1}}
\put(6,7){\line(-1,-2){1}}
\put(7,5){\makebox(0,0){$\bullet$}}
\put(6,7){\makebox(0,0){$\bullet$}}
\put(6,6){\makebox(0,0){$\bullet$}}
\put(5.5,5.5){\makebox(0,0){$\bullet$}}
\put(8,6){\makebox(0,0){$\bullet$}}
\put(7,5){\makebox(0.6,-0.5){$a(1)$}}
\put(6,7){\makebox(0.9,0){$a(2)$}}
\put(6,6){\makebox(0.4,-0.4){$a(3)$}}
\put(5.5,5.5){\makebox(0.6,-0.4){$a(4)$}}
\put(8,6){\makebox(0.6,-0.4){$a(5)$}}
\put(6,3.5){\makebox(0,0){Figure 1}}
\end{picture}$$

Under these two conditions we have that $Y=\{x_n=0\}$ is an irreducible component of the singular locus of $M_A (\beta)$ \cite[Sec.~3]{Adolphson}, there is only one slope of $M_A(\beta)$ along $Y$ \cite{schulze-walther} and, if $\beta$ is generic enough, the dimension of the space of Gevrey series solutions of $M_A(\beta)$ along $Y$ is $d! {\rm Vol}(\Delta_\sigma)$ \cite{F}.

We prove in Theorem \ref{Asymptotic-expansions-theorem} that for generic $\beta \in \CC^d$, the space of Gevrey series solutions of $M_A(\beta)$, along the hyperplane $Y$, has a basis given by asymptotic expansions of certain holomorphic solutions of $M_A(\beta)$ described by integral representations, as those considered by Adolphson in \cite[Sec.~2]{Adolphson}. These integrals are solutions of type $$I_C(\beta; x)= I_C(\beta; x_1,\ldots,x_n):=\int_C t^{-\beta-\nn }\exp\left(\sum_{\ell=1}^n x_\ell t^{a(\ell)}\right)dt$$ where $t=(t_1,\ldots,t_d)$, $dt=dt_1\cdots dt_d$ and $C$ runs over a carefully chosen and explicit finite set of cycles on the universal covering of $(\CC^*)^d$. Moreover, we also prove in Theorem \ref{conclusion-theorem} that these cycles, which are Borel--Moore cycles on the universal covering of $(\CC^*)^d$, can be replaced by a set of rapid decay homology cycles in the sense of \cite{Hien09}.

Here is a summary of the content of this paper. In Section \ref{intro} we consider a general
matrix $A$ as before but not necessarily satisfying previous conditions (1) and (2) (see Assumption \ref{assumption}). Following a construction
in \cite[Sec.~4.4]{GG99}, we describe cycles $C_{p,\delta}$ in the universal covering of
$(\CC^*)^d$, depending on a given point $x\in \CC^n$. We fix a maximal simplex $\sigma\subset\{1,\dots,n\}$, i.e. the set $\{a_k\,\vert \, k\in \sigma\}$ is a basis of $\RR^d$. Then this cycle depends only on $x_\sigma:=(x_k)_{k\in\sigma}$, and on vectors
$p\in \ZZ^\sigma$ and $\delta \in \RR^\sigma$ with components $\delta_k$ satisfying $|\delta_k|<1/2$.
In Subsection \ref{sufficient-cond-mod-growth} we give a sufficient condition for the argument of the integral
$I_{p,\delta}(\beta; x):=  I_{C_{p,\delta}}(\beta; x)$ to have moderate growth along
$C_{p,\delta}$, with a bounded exponential factor. This is a step towards sufficient conditions of convergence for $I_{p,\delta}(\beta; x)$ which are developed
in Section \ref{section-change-coordinates}.

In Section \ref{section-change-coordinates}, we perform the appropriate toric change of variables in the universal covering of $(\CC^*)^d$, like in \cite{GG99}, which reduces  the description of asymptotic expansions for the integrals $I_{p,\delta}(\beta; x)$ to the study of integrals of type $$ F_{p,\delta}(\beta; y):= \int_{D_{p,\delta}} t^{-\beta-\bf{1}}\exp\left(t_1+\cdots+t_d+\sum_{j=d+1}^ny_jt^{a(j)}\right)dt$$ where the cycle  $D_{p,\delta}$ is the image of $C_{p,\delta}$ under the change of variables. We notice that  after this reduction the  new integral looks like a particular case of $I_{p,\delta}(\beta; x)$, with the first $d\times d$ submatrix $(a(1),\ldots,a(d))=(e(1),\ldots,e(d))$ equal to the identity matrix. However, the matrix $A=(a(1),\ldots,a(n))$ is now allowed to have rational non integer coefficients.

The crucial point for convergence statements is a condition of rapid decay at infinity as written in the inequality \eqref{dominant-exponent}.
We prove that, under some conditions, the integral $F_{p,\delta}(\beta; y)$ is absolutely convergent when   $\Re\beta_k<0$ for $k\in \sigma$ and
$y \in (\CC^*)^{n-d}$; see Lemmata \ref{conv}, \ref{sufficientplus}.

In Section \ref{obtention-Gevrey-series} we show that these convergence statement can be applied in practise: under the
Assumption \ref{assumption-2}, and with a careful choice of the parameter $\delta$, depending on $p\in \ZZ^\sigma$ and  $x\in (\CC^*)^n$, we obtain an effective statement of convergence in Lemma \ref{remark-delta}.	

Section \ref{obtention-Gevrey-series} contains some of the main results of this paper. We assume in this section that the matrix $A$ defined in Section \ref{section-change-coordinates}, satisfies moreover conditions (1) and (2) (see Assumption \ref{assumption-2}), deduced from the condition (1) and (2) in Assumption \ref{assumption} already considered for the original matrix.

We fix $p\in \ZZ^d$ and $\delta\in \RR^d$ once for all and we omit these subindexes  in our formulas. As a step towards previously mentioned Theorem  \ref{Asymptotic-expansions-theorem}, we prove in Theorem \ref{ASY} that if $\Re \beta <0$, there is an asymptotic expansion with respect to the variable $y_n$ in some sector in $\CC^*$:
\begin{equation}\label{F-a-exp} F(\beta; y) \, \underset{y_n\to 0} {\sim} \, \sum_{m\in \NN}A(\beta; m, y')\frac{y_n^m}{m!}
\end{equation}
where $y'=(y_{d+1},\ldots, y_{n-1})$ and $$A(\beta; m , y'):=\int_{D_{p,\delta}}t^{-\beta-\nn+ma(n)}\exp\left(t_1+\cdots+t_d + \sum_{j=d+1}^{n-1} y_j t^{a(j)}\right) dt.$$
Assumption \ref{assumption-2} plays an essential role in the proof of this result. Without assumption (1), we would need to impose further conditions on the arguments of $y$, see Remark \ref{remark-first-condition}, in order to guarantee the convergence of $F(\beta; y)$. Without condition (2), the vertex $a(n)$ has negative components and the integrals defining the coefficients $A(\beta; m, y')$ already fail to be convergent for $m$ large enough.

Then we prove in Lemma \ref{analytic-continuation} that $F(\beta; y)$ admits a meromorphic continuation $\widetilde{F}(\beta; y)$, with respect to the variable $\beta$, with poles at most in a countable locally finite union of hyperplanes $\mathcal P$ in $\CC^d$. The proof of this lemma uses that the points $a(d+1),\ldots, a(n)$ belong to $\sum_{k=1}^d \RR_{> 0} a(k)=\RR^d_{>0}$, which follows from conditions (1) and (2). The set $\mathcal P$ is contained in the set of so-called \emph{resonant parameters} of $A$ \cite[2.9]{GKZ90}
and it is explicitly described in terms of the columns of $A$.
We also prove in Lemma \ref{asymptotic-expansion-tilde} that, for any fixed parameter $\beta \not\in {\mathcal P}$, the meromorphic continuation  $\widetilde{F}(\beta; y)$ admits an asymptotic expansion along $y_n=0$ and that the coefficients $\widetilde{A}(\beta; m, y')$ of this expansion are the analytic continuation of the previously introduced $A(\beta; m, y')$.

In Section \ref{Limit-of-integrals-rapid-decay} we prove that when $\Re \beta <0$ and $\beta$ is sufficiently general, the integrals $F(\beta;y)$ are in fact equal to integrals over rapid decay cycles in the sense of \cite{Hien09} (see Theorem \ref{limit-rapid-decay}). The statements involving Borel-Moore cycles are weaker because the analytic continuations are not expressed by integral along cycles when $\Re \beta_k>0$ for some $k$. The notion of rapid decay cycles is explained in Subsection \ref{Hiensufficient0}. Subsection \ref{the rd cycle} is devoted to the construction of rapid decay cycles. We start from a product of Hankel contours, along which the hypergeometric integrals are grossly divergent, but then we build a thinned towards infinity version of this product along which convergent integrals are obtained. These integrals in Section \ref{Limit-of-integrals-rapid-decay} are also defined when $\Re \beta_ k \geq 0$ for some $k$ and they are still solutions of $M_A(\beta)$. In Subsection \ref{the rd cycle} we prove, by using Section \ref{obtention-Gevrey-series}, that these integrals admit asymptotic expansions as Gevrey series solutions of $M_A(\beta)$ for non resonant $\beta$ in $\CC^d$.

Acknowledgements: We would like to thank K. Takeuchi and S.-J. Matsubara-Heo for their suggestions and useful comments about the content of this article. The first author would like to thank the D\'epartement de Math\'ematiques of the University of Angers (France) for its support during the first stage of this research. The third author would like to thank the Department of Algebra and the Institute of Mathematics of the University of Seville (IMUS) for their support and hospitality during the preparation of this paper.

\section{Products of lines for rapid decay.}\label{intro}
\subsection{Notations}

Let us slightly change our notation used in the introduction and let us start with a pair $(B,\gamma)$ where $B:=\left(b(1),\dots,b(n)\right)\in \left(\ZZ^d\right)^n$ is a $d\times n$ matrix, described as a list of columns such that $\ZZ B:=\ZZ b(1)+\cdots+ \ZZ b(n)=\ZZ^d$ and where  $\gamma$ is a parameter vector  in $\CC^d$. We are concerned with integrals:
$$I_C(\gamma; x)=I_C(\gamma; x_1,\ldots,x_n):=\int_C t^{-\gamma-\nn }\exp\left(\sum_{\ell=1}^n x_\ell t^{b(\ell)}\right)dt$$
where $\nn=(1,\ldots,1)\in \NN^d$ and $C$ is a suitable cycle. These integrals are formally solutions of the GKZ system $M_B (\gamma)$ associated with $(B,\gamma)$ (see e.g. \cite[Sec.~2]{Adolphson}).

To make precise this definition let us specify that we use, throughout the paper, the following conventions and notations.

First, $C$ is a cycle on the universal covering $(\wt{\CC^*})^d$ of $(\CC^*)^d$. We identify
$(\wt{\CC^*})^d$ with $\CC^d$ or with $\RR_{>0}^d\times \RR^d$ and write $z=(\log r+\sqrt{-1}\,\theta)$ or $(r,\theta)$ respectively, for the coordinates on $(\wt{\CC^*})^d$  with $\theta_k$ a determination of $\arg t_k$, $t_k=\exp(z_k)$ and $r_k=|t_k|$. We set, for any vector $v\in \CC^d$, $t^v=\prod_{k=1}^d t_k^{v_k}$.
This is a multivalued monomial, namely the function on the universal covering:
$$\exp{\ideal{z,v}}=\exp\left(\sum_{k=1}^d v_k(\log r_k+\sqrt{-1}\,\theta_k)\right)$$
where we set, given two vectors $u,v \in \CC^d$, $\ideal{u,v}=\sum_{k=1}^d u_k v_k$.

We are interested in cycles $C$ in $(\widetilde{\CC^*})^d$, such that the integrals $I_C(\gamma; x)$ are convergent and have asymptotic expansions along a fixed coordinate hyperplane. We want to find sufficiently many cycles $C$ so that these asymptotic expansions form a basis of the space of Gevrey solutions of $M_B (\gamma)$.
We achieve this goal only under some assumptions on $B$ and $\gamma$ (see Theorem \ref{Asymptotic-expansions-theorem}).

\subsection{Description of cycles of rapid decay at infinity}
If $\tau \subset  \{1,\ldots,n\}$, we denote by $B_\tau$ the matrix whose columns are $b(j)$ with $j\in \tau$ and by $\ol{\tau}$ the complement of $\tau$ in $\{1,\ldots,n\}$.

Recall that a subset $\sigma \subset \{1,\ldots,n\}$ is called a maximal simplex for $B$ if the columns $\{b(k), k\in \sigma\}$ form a basis of $\RR^d$. Such a maximal simplex $\sigma$ is also called a base in \cite[Sec.~1.1]{GZK89}.
We often identify the set $\sigma$ with the set of columns $\{b(k), k\in \sigma\}$.

We fix a maximal simplex $\sigma$ for $B$ and take $x\in\CC^n$ such that $x_k\neq 0$ for all $k\in\sigma$. We also fix $p=(p_k )_{k\in \sigma} \in \ZZ^{\sigma}\simeq \ZZ^d$, $\delta =(\delta_k)_{k\in\sigma} \in \RR^{\sigma}\simeq \RR^d$
such that $|\delta_k |<\frac{1}{2}$ for all $k\in \sigma$. We denote by $C_{p,\delta}$ the cycle in the space $(\widetilde{\CC^*})^d$ described by the following condition on the argument $\theta:=\arg t$ of $t\in (\CC^*)^d$ (i.e.\,
$^t\theta: = (\arg t_1,\ldots,\arg t_d)$):
\begin{equation}\label{cycleCp-delta}
\arg(x_k t^{b(k)})=\arg x_k+\ideal{b(k),\theta}=(1+\delta_k+2p_k)\pi \quad \text{ for all } k\in \sigma.
\end{equation} Notice that $C_{p,\delta}$ depends also on $x_{\sigma}:=(x_k )_{k\in \sigma} \in (\CC^*)^\sigma\simeq (\CC^*)^d$.

From now on we will denote $I_{p,\delta}(\gamma ; x)= I_{C_{p,\delta}}(\gamma ; x)$. The cycles $C_{p,\delta}$ are a slightly  modified version of cycles considered in \cite[Sec.~4.4]{GG99}.

Let us set  $\Theta := \left]\frac{\pi}2,\frac{3\pi}2\right[+2\pi \ZZ$. The equality (\ref{cycleCp-delta}) can be globally rewritten using matrix notation:
\begin{equation*}
\arg x_{\sigma}+~
^tB_\sigma \theta =(\nn+\delta+2p)\pi \in \Theta^{\sigma}.
\end{equation*}
There is a unique solution $\theta$ of the previous equation
\begin{equation}\label{arg}
\theta= (~^tB_\sigma)^{-1}\left(-\arg x_\sigma + (\nn+\delta+2p)\pi  \right)
\end{equation}
so that $C_{p,\delta}$ is the cartesian product of $d$ open half--lines.

Given $p,p'\in \ZZ^d$, let $\theta=\arg t$, $\theta'=\arg t'$ be the corresponding unique solutions for equation (\ref{arg}).

If $(~^tB_\sigma)^{-1}(p-p') \in \ZZ^d$ then $\theta - \theta' \in 2\pi \ZZ^d$ and the projections of the two cycles $C_{p,\delta}$ and  $C_{p',\delta}$ on $(\CC^*)^d$ are the same.
We check that the convergence of the two integrals along the cycles $C_{p,\delta}$ and  $C_{p',\delta}$  are then equivalent to each other and, moreover, the integral solutions differ only by a constant factor:
\[
I_{p',\delta}(\gamma; x) = \int_{C_{p',\delta}}t^{-\gamma-\nn}\exp\left(\sum_{\ell=1}^n x_\ell t^{b(\ell)}\right)dt = e^{-2\pi \sqrt{-1} (\sum_{k=1}^d m_k\gamma_k)} I_{p,\delta}(\gamma ; x)
\] for some $m_k\in \ZZ$, $k=1,\ldots,d$.

When $p$ varies in a set of representatives of $\frac{\ZZ^d}{\ZZ~^tB_\sigma}$, we will see that the convergence of
the integral $I_{p,\delta}(\gamma; x)$ depends on $\delta$ (see Remark \ref{remark-arg} and Lemma \ref{remark-delta}).
However, choosing in each such class an appropriate $\delta $, we can find, as a consequence of our main result and under some conditions (see Assumption \ref{assumption}),
$[\ZZ^d:\ZZ~^tB_\sigma]=\vert\det B_\sigma\vert$  many integral solutions $I_{p,\delta}(\gamma; x)$ which
are linearly independent (see Theorem \ref{Asymptotic-expansions-theorem}).

We will see in the proof of Lemma \ref{sufficientplus}, after the change of variables defined in Section \ref{section-change-coordinates},  that the cycles $C_{p,\delta}$ are of rapid decay at infinity.

\subsection{Sufficient conditions for moderate growth}\label{sufficient-cond-mod-growth} Sufficient conditions for the convergence of the integral  $I_{p,\delta}(\gamma; x)$ are detailed in next Section (see Lemma \ref{sufficientplus} and Remark \ref{sufficientdoubleplus}). As a preliminary step let us look here at a condition for bounding the exponential term in that integral:

Let us notice that condition (\ref{cycleCp-delta}) implies that $\Re x_kt^{b(k)}<0$ along  $C_{p,\delta}$ for any $x_k\in \CC^*$, $k\in \sigma$. If we assume the analogous condition
\begin{equation}\label{argxj}
\arg(x_jt^{b(j)})\in \Theta \quad \text{ for all } j\in\ol{\sigma}
\end{equation} then the argument of the exponential has real negative part everywhere along $C_{p,\delta}$, hence the absolute value of the exponential term in the integral $I_{p,\delta}(\gamma;x)$ is bounded by 1.
Conditions (\ref{argxj}) can be globally rewritten:
\begin{equation}\label{arg2}
\arg x_{\ol{\sigma}}+~^tB_{\ol{\sigma}}\, \theta \in \Theta^{\ol{\sigma}}.
\end{equation}
Finally we have, if we take into account the term $t^{-\gamma-{\bf 1}}$:
\begin{remark} Under the conditions (\ref{arg}) and (\ref{arg2}), the argument of the integral $I_{p,\delta}(\gamma; x)$ has moderate growth along $C_{p,\delta}$.
\end{remark}
We will see in the next section that conditions (\ref{arg}) and (\ref{arg2}) are sufficient  convergence conditions for the integrals $I_{p,\delta}(\gamma; x)$ when combined with a condition on the parameter $\gamma$. After an appropriate change of variables we can interpret them as a condition of rapid decay at infinity, see the proof of Lemma \ref{sufficientplus}.

\begin{remark}\label{remark-arg}
Let us notice that condition \eqref{arg} determines a unique cycle $C_{p,\delta}$. It is not clear that for given $x\in (\CC^*)^\sigma\times \CC^{\ol\sigma}$ and $p\in \ZZ^\sigma$
one can always choose $\delta\in \RR^\sigma$ for this cycle to satisfy as many conditions as in expression \eqref{arg2}. It is therefore interesting to weaken these conditions, by eliminating non significant ones, as we shall do in Lemma  \ref{sufficientplus} and Remark \ref{sufficientdoubleplus}.\end{remark}

We notice that $C_{p,\delta}$ is a Borel--Moore cycle in $(\wt{\CC^*})^d$
but not in general a rapid decay cycle in the sense of \cite{Hien07}, see Remark \ref{no-rd-at-0}.

However we shall prove in Section \ref{Limit-of-integrals-rapid-decay} that the integral along $C_{p,\delta}$ is equal to an integral along a rapid decay cycle (see Theorem \ref{limit-rapid-decay}) under Assumption \ref{assumption-2}, and for values of $\gamma$ which guarantee convergence. Our result can be then interpreted in the frame of \cite[Th.~4.5]{Esterov-Takeuchi-2015}.

\section{A change of variables and explicit calculations.}\label{section-change-coordinates}

We will assume for simplicity, after a possible reordering of the variables, that the simplex $\sigma$ is $\{1,\ldots,d\}$. Let us fix $x\in (\CC^*)^d\times\CC^{n-d}$.
We make the following toric change of variables, which is well defined and $\CC$-linear on the universal covering $(\wt{\CC^*})^d\simeq \CC^d$:
 \[
s_k=x_{k}t^{b(k)} \text{ for } k \in \sigma.
 \]
Equivalently, we have $$t_k = \left(\frac{s}{x_\sigma}\right)^{B_\sigma^{-1}e(k)} \text{ for } k \in \sigma$$ where
$\left(\frac{s}{x_\sigma}\right)$ is the vector with coordinates $s_k/x_k$  and $e(k)$ is the $k$--th standard basis vector in $\ZZ^d$, for $k\in \sigma$.

Let us compute the jacobian matrix and determinant of this change of variables.
\[
\frac{\p t_k}{\p s_j}=\left(B_\sigma^{-1} e(k)\right)_j\, s_j^{-1}\left(\frac{s}{x_\sigma}\right)^{B_\sigma^{-1}e(k)}.
\]
Since $\left(B_\sigma^{-1} e(k)\right)_j=\left(B_\sigma^{-1}\right)_{j,k}$, we rewrite this:
\[
J:=\left(\frac{\p t_k}{\p s_j}\right)_{1\leq k,j\leq d}=\left(\begin{array}{ccc}
\left(\frac{s}{x_\sigma}\right)^{B_\sigma^{-1}e(1)}&  & 0 \\
 &\ddots & \\
0& & \left(\frac{s}{x_\sigma}\right)^{B_\sigma^{-1}e(d)}
\end{array}\right)
\left(~^tB_\sigma\right)^{-1}
\left(\begin{array}{ccc}
\frac1{s_1}&  & 0 \\
 &\ddots & \\
0& & \frac1{s_d}
\end{array}\right).
\]
The integral is transformed as follows:
\begin{align*}
 I_{p,\delta}(\gamma ;x)=&\int_{C_{p,\delta}} t^{-\gamma-\nn}\exp\left(\sum x_\ell t^{b(\ell)}\right)dt = \\
= &\int_{D_{p,\delta}}\det(J) \left(\frac{s}{x_\sigma}\right)^{-B_\sigma^{-1}(\gamma+\bf{1})}\exp\left(\sum_{k\in \sigma} s_k + \sum_{j\notin\sigma} x_jx_\sigma^{-B_\sigma^{-1}b(j)}s^{B_\sigma^{-1}b(j)}\right) ds.
\end{align*}
Since we have $\det(J)=\left(\frac{s}{x_\sigma}\right)^{B_\sigma^{-1} \bf{1}}\det(B_\sigma^{-1})\, s^{-\bf{1}}$, the final result is:
\[
I_{p,\delta}(\gamma; x)=\det(B_\sigma^{-1}){x_\sigma}^{B_\sigma^{-1}\gamma}\int_{D_{p,\delta}} s^{-B_\sigma^{-1}\gamma - \nn} \exp\left(\sum_{k\in\sigma} s_k + \sum_{j\notin\sigma} x_jx_\sigma^{-B_\sigma^{-1}b(j)}s^{B_\sigma^{-1}b(j)}\right)ds.
\]
The cycle $D_{p,\delta}$  is the image of the cycle $C_{p,\delta}$ described in Section \ref{intro}, and it is determined by the conditions deduced from equality (\ref{cycleCp-delta}):
\begin{equation}\label{Dpdelta}
\arg s_k =(1+\delta_k+2 p_k)\pi \quad \mbox{ for all } k\in \sigma.
\end{equation}
\begin{remark}\label{finite_covering}
Let us notice that the argument of the exponential term in previous integral is already defined as a univalent polynomial function on a finite covering of $(\CC^*)^d$, isomorphic to $(\CC^*)^d$. More precisely we ramify $q_k$ times the factor $\CC^*$ of the  variable $s_k$ with $q_k$ the
lowest common denominator of the coefficients in the $k$--th row of the matrix $B_\sigma^{-1}B$.
This will be used in Section \ref{Limit-of-integrals-rapid-decay} and especially in Remark \ref{Hiencycle}.
\end{remark}
\begin{lemma}\label{conv} Sufficient conditions for the absolute convergence of $I_{p,\delta}(\gamma; x)$ for $x\in(\CC^*)^n$ are:
\[
\Re B_\sigma^{-1}\gamma<0 \,\, \text{ and } \,\, \Re (x_jx_\sigma^{-B_\sigma^{-1}b(j)}s^{B_\sigma^{-1}b(j)}) < 0\quad \forall j\in \ol{\sigma}, \, \, \forall s\in D_{p,\delta}.
\]
\end{lemma}

\begin{proof}
The second condition is a direct translation of (\ref{argxj}). It is sensitive to the choice modulo $2\pi\ZZ$ of $\arg s_k$, for $k\in \sigma$,  since the matrix $B_\sigma^{-1}b(j)$ may have coefficients in $\QQ\setminus \ZZ$.

We have $s_k=-|s_k|e^{\sqrt{-1}\,\pi \delta_k}$ and $(|s_1|,\dots,|s_d|)\in \RR_{>0}^d$ parametrizes the cycle. Since all the terms in the argument of the exponential term in the integral have real negative part, we have
\[
\Re\left(\sum_{k\in \sigma} s_k + \sum_{j\notin\sigma} x_jx_\sigma^{-B_\sigma^{-1}b(j)}s^{B_\sigma^{-1}b(j)}\right)\leq \Re\left(\sum_{k\in \sigma} s_k\right)\leq -c\left(\sum_{k\in \sigma} |s_k|\right)
\]
with $c:=\min_{k}\cos(\pi\delta_k)>0$.

Therefore the integral $I_{p,\delta}(\gamma; x)$ is dominated by the following convergent integral with $\alpha_k=\Re(-(B_\sigma^{-1}\gamma)_k)-1>-1$\,:
\[
\int_{\RR_{>0}^d} r^{\alpha}\exp\left(-c\sum_{k\in \sigma} r_k\right)dr= c^{-|\alpha|-d} \Gamma (\alpha +{\bf 1}).
\] where $r=(r_1,\ldots,r_d)$, $|\alpha|=\alpha_{1}+\cdots +\alpha_d$ and $\Gamma (\alpha + {\bf 1})= \prod_{k=1}^d \Gamma ( \alpha_k + 1).$
\end{proof}

We now define a reduction of the integral $I_{p,\delta} (\gamma ; x)$, which contains all the essential information. We put aside the initial monomial ${x_\sigma}^{B_\sigma^{-1}\gamma}$, and the constant $\det(B_\sigma^{-1})$, and the remaining integral can be expressed via a function of $n-d$ variables $y$ indexed by $\ol{\sigma}$:
\begin{equation}\label{reducedI-0}
G_{p,\delta}(\gamma; y)=G_{p,\delta}(\gamma; y_{d+1},\ldots,y_n) := \int_{D_{p,\delta}} s^{-B_\sigma^{-1}\gamma-\bf{1}} \exp\left(\sum_{k\in \sigma} s_k + \sum_{j\notin\sigma} y_js^{B_\sigma^{-1}b(j)}\right) ds.
\end{equation}

The formula relating $I_{p,\delta}(\gamma , x)$ to previous integral is:

\begin{equation}\label{conv-I-to-F}  I_{p,\delta}(\gamma, x)=\det(B_\sigma^{-1}){x_\sigma}^{B_\sigma^{-1}\gamma}G_{p,\delta}(\gamma;y)\end{equation}
where $y_j= x_jx_\sigma^{-B_\sigma^{-1}b(j)}$ for $j\in \ol{\sigma}.$

\subsection{Reduced version of hypergeometric integrals}\label{reduced-version-F}
In order to simplify subsequent calculations we shall use a more handy version of the integral $G_{p,\delta}(\gamma; y)$ by renaming the exponents.

We consider $\beta\in \CC^d$ and a $d\times n$ matrix $A=(a(1),\dots,a(n))$ with rational coefficients and with $(a(1),\dots,a(d))=(e(1),\dots,e(d))$ the unit matrix. We define
\begin{equation}\label{reducedI}
F_{p,\delta}(\beta; y):= \int_{D_{p,\delta}} t^{-\beta-\bf{1}}\exp\left(t_1+\cdots+t_d+\sum_{j=d+1}^ny_jt^{a(j)}\right)dt
\end{equation}
and we recover $G_{p,\delta}$ by setting $A=B_{\sigma}^{-1}B$ and $\beta =B_{\sigma}^{-1}\gamma$ in $F_{p,\delta}$.
We also replace the variable $s$ by the variable $t$ of the beginning of Section \ref{intro}.

Now we transpose to $F_{p,\delta}(\beta; y)$ the two sufficient conditions in Lemma \ref{conv}. The first one simply becomes $\Re\beta_k<0$, for all $k\in \sigma$,
and we shall assume it until the end of this section. Then we focus on the second condition in Lemma \ref{conv}. This condition transposed to $F_{p,\delta}(\beta; y)$ is:
\begin{equation}\label{sufficient}
\arg(y_j t^{a(j)})= \arg y_{j}+ \ideal{{\bf{1}}+\delta + 2p,\, a{(j)}}\pi \in  \Theta \quad \text{ for } j\in \ol{\sigma}.
\end{equation}
We shall prove in the next two lemmas, that a part of conditions $(\ref{sufficient})_\ell$ is already sufficient to guarantee the convergence of the integral $F_{p,\delta}(\beta; y)$.

We set $y_k=1$ for $k=1,\dots,d$. Notice that condition $(\ref{sufficient})_k$ is satisfied for all $k\in \sigma$ because $t_k = y_k t^{a(k)}$  and $|\delta_k|<1/2$. So, condition (\ref{sufficient}) is equivalent to
\begin{equation}\label{sufficient-bis}
\arg(y_\ell t^{a(\ell)})= \arg y_{\ell}+ \ideal{{\bf{1}}+\delta + 2p,\, a{(\ell)}}\pi \in  \Theta \quad \text{ for } \ell \in \{1,\ldots,n\}.
\end{equation}

Notice that in condition $\eqref{sufficient-bis}_\ell$ we implicitly assume that $y_\ell\neq 0$.

Recall that $\Delta_A$ denotes the convex hull of $\{{\bf{0}}, a(1),\ldots,a(n)\}$ in $\RR^d$. Let us denote by $\p\Delta_A$ the union of facets of  $\Delta_A$ not containing the origin and by $\tau_A$ the set of indices $\ell\in\{1,\ldots,n\}$ such that $a(\ell)\in \p\Delta_A$. We denote by $\eta\subset\tau_A$ the set of indices for the vertices of $\Delta_A$ different from the origin.
Recall also that we have set $\sigma = \{1,\ldots,d\}$. Notice that the intersection $\eta \cap \sigma$ could be non empty. Finally, let $M(t,y)$ denote the argument in the exponential term in the integral (\ref{reducedI}).

\begin{lemma}\label{sufficientplus}
The set of conditions $(\ref{sufficient-bis})_\ell$
for $\ell$ in $\tau_A$ is sufficient for the integral $F_{p,\delta}(\beta; y)$ to be absolutely convergent, when $\Re\beta<0$.
\end{lemma}

\begin{proof}
Since $|\exp(M(t,y))|=\exp(\Re M(t,y))$ we will provide first a bound of $\Re M(t,y)$.
We set $\tau:= \tau_A \cup \sigma$. We can write, for $j \not\in \tau$,
\begin{equation}\label{a(j)}
a(j)=\sum_{\ell \in\eta} \nu_{j\ell}a(\ell)
\end{equation}
with  $0\leq\sum_{\ell\in\eta} \nu_{j\ell}<1$  { and } $\nu_{j\ell}\geq 0$ for all $\ell\in \eta$ and $j\notin \tau$.

We set
\[
\xi_\ell:=|y_{\ell}t^{a(\ell)}|\, {\text{ for }}\, \ell \in \tau.
\]
By the condition $(\ref{sufficient-bis})_\ell$ for  $\ell \in \tau$ there exists $\vartheta\in]0,\frac{\pi}2[$ such that for all $\ell \in \tau $ one has
$$
\Re (y_{\ell}t^{a(\ell)}) \leq -\xi_\ell\cos\vartheta.
$$

Recall that
\begin{align*}
M(t,y) = &\sum_{\ell \in \tau}  y_{\ell}t^{a(\ell)}+\sum_{j\notin\tau}y_jt^{a(j)}=\sum_{\ell \in \tau}  y_{\ell}t^{a(\ell)}+\sum_{j\notin\tau}y_j\prod _{\ell \in \eta} \left(t^{a(\ell)}\right)^{\nu_{j\ell}} = \\
=& \sum_{\ell\in \tau} y_{\ell}t^{a(\ell)}+\sum_{j\notin\tau}\frac{y_j}{\prod_{\ell \in \eta} y_{\ell}^{\nu_{j\ell}}}\prod _{\ell \in \eta} \left(y_{\ell}t^{a(\ell)}\right)^{\nu_{j\ell}}.
\end{align*}

Therefore we get
\begin{align}\label{Re-exponent}
\Re M(t,y)\leq & -\left(\sum_{\ell \in \tau} \xi_\ell\right)\cos\vartheta +  \sum_{j\notin\tau}\left| \frac{y_j}{\prod_{\ell \in \eta}  y_{\ell}^{\nu_{j\ell}}}\right|(\max_{\ell\in \eta} \xi_\ell)^{\sum_{\ell\in \eta} \nu_{j\ell}} \leq \\
\leq & -\left(\sum_{\ell \in \tau} \xi_\ell \right)\cos\vartheta + K\max\left(\left(\sum_{\ell \in \tau} \xi_\ell\right)^{\hspace{-1mm}\kappa},1\right)
\end{align}
where $\kappa:=\max_{j \notin \tau} (\sum_{\ell \in \eta} \nu_{j\ell})<1$, and $K=\sum_{j\not\in \tau}\left| \frac{y_j}{\prod_\ell y_{\ell}^{\nu_{j\ell}}}\right|$ is a constant for a fixed value of $y\in (\CC^*)^n$. Set $\xi:=\sum_{\ell\in \tau} \xi_\ell$. We see that $\Re M(t,y)$ is bounded by $- \xi\cos \vartheta+K\xi^\kappa $ which tends to $-\infty$ when $\xi \to +\infty$.

It also follows that $|\exp(M(t,y))|$ is bounded:  $$|\exp(M(t,y))|\leq \exp(- \xi\cos \vartheta+K\max(\xi,1)^\kappa )\leq e^{L}$$ where $L:=\sup_{\xi\in\RR_{>0}}(-\xi\cos\vartheta+K\max(\xi,1)^\kappa).$

Since $\sigma \subseteq \tau$, we have $\xi=\left(\sum_{\ell\in \tau} \xi_\ell\right)\geq |t_1|+\cdots+|t_d|$
and we get a rapid decay condition when $|t_1|+\cdots+|t_d| \to +\infty$. There are positive constants $c>0$ small enough and $R>0$ big enough such that
\[
|t_1|+\cdots+|t_d|>R \Longrightarrow|\exp(M(t,y))|<\exp(-c(|t_1|+\cdots+|t_d|)).
\]
It is convenient for further calculation to incorporate the upper bound $e^{L}$ in a global inequality. For  $C=L+c R>0$ we have that
\begin{equation}\label{dominant-exponent}
\forall t\in D_{p,\delta}, \quad |\exp(M(t,y))|<\exp(C-c(|t_1|+\cdots+|t_d|)).
\end{equation}
The absolute convergence of the integral $F_{p,\delta}(\beta; y)$ follows now exactly as in the proof of Lemma \ref{conv} by the assumption $\Re \beta_k <0$ for all $k\in  \{1,\ldots ,d\}$. \end{proof}

\begin{remark} Notice that for fixed $p\in \ZZ^d$ and $\delta\in \RR^d$ with $\vert \delta_k \vert < 1/2$, the set of conditions $\eqref{sufficient-bis}_\ell$ on $y=(y_{d+1},\ldots,y_n)$, for $\ell \in \tau_A$, defines
an open set in $\CC^{\{d+1,\dots,n\}\setminus\tau_A}\times(\CC^*)^{\tau_A\setminus\sigma}$.
On the factor $(\CC^*)^{\tau_A\setminus\sigma}$ this open set is a product of open sectors. \end{remark}

\begin{remark}\label{uniform-constants}
In the proof of Lemma \ref{sufficientplus}, if we assume that $y$ varies in a compact neighborhood of a given point in $(\CC^*)^{n-d}$,
we can take the constants $K,L,c,R,C$ (which depend on $y$) as uniform bounds with respect to $y=(y_{d+1},\ldots,y_n)$.\end{remark}

\begin{remark}\label{no-rd-at-0} Notice that, in general,  we don't have rapid decay at the origin. For example, if the matrix $A$ has only positive entries, the exponential term in the integral $F_{p,\delta}(\beta; y)$ is continuous and tends to $1$, when $|t_1|+\cdots+|t_d|\to 0$, and the integrand behaves at the origin as the factor $t^{-\beta-\nn}$.
\end{remark}

Recall that $\eta$ is the set of vertices of $\Delta_A$ different from the origin. We may  weaken the hypothesis in Lemma \ref{sufficientplus} as follows: we have an analogous formula to (\ref{a(j)}) for all $j\notin \eta \cup \sigma$, namely $a(j)=\sum_{\ell \in\eta} \nu_{j\ell}a(\ell)$, with $\kappa_j:=\sum_\ell \nu_{j\ell}\leq 1$. Precisely $\kappa_j=1$ for $j\in \tau \setminus (\eta \cup \sigma)$ and $\kappa_j <1$ for $j\notin \tau=\tau_A \cup \sigma$. We set
$K_j=\left| \frac{y_j}{\prod_{\ell \in \eta} y_{\ell}^{\nu_{j\ell}}}\right|$ for $j\notin \eta \cup \sigma$ and we obtain
\begin{remark}\label{sufficientdoubleplus}
The set of conditions $(\ref{sufficient-bis})_\ell$  for $\ell\in \eta$, is sufficient for the convergence of $F_{p, \delta}(\beta; y)$ when $y$ varies in the non empty open set in $(\CC^*)^{n-d}$ defined by $\sum_{j\in \tau \setminus (\eta \cup \sigma)} K_j<\cos\vartheta$.

More precisely we can write down a refined upper bound of the real part of the exponent:
\[
\Re M(t,y)\leq-\left(\sum_{\ell\in \eta \cup \sigma}\xi_\ell\right)\cos\vartheta +  \sum_{j\notin\tau}K_j\left(\sum_{\ell\in \eta \cup \sigma}\xi_\ell\right)^{\hspace{-1mm}\kappa_j}+\sum_{j\in\tau \setminus (\eta \cup \sigma)} K_j\left(\sum_{\ell \in \eta \cup \sigma} \xi_\ell\right)^1
\]
with $K_j=\left| \frac{y_j}{\prod_{\ell \in \eta} y_{\ell}^{\nu_{j\ell}}}\right|$, for $j\notin\eta \cup \sigma$, and $\kappa_j=\sum_k \nu_{jk}<1$, for $j\notin\tau$.
Thus, the conclusion follows by an argument similar to the one in Lemma \ref{sufficientplus}.
\end{remark}

\section{Obtention of the Gevrey series.}\label{obtention-Gevrey-series}
Given a matrix $B=(b(1),\ldots, b(n))$ as in Section \ref{intro} and $\tau\subset \{1,\ldots,n\}$ we denote by $\Delta_\tau$ the convex hull of $\tau$ and the origin. We assume, after a possible reordering of the variables, that $\sigma:=\{1,\ldots,d\}$ is a maximal simplex for $B$, i.e. $\{b(1),\ldots, b(d)\}$ is a basis of $\RR^d$.

The main result in this section is Theorem \ref{Asymptotic-expansions-theorem}. We prove  that
for generic $\gamma \in \CC^d$ and under some assumptions on $B$, the space of Gevrey series solutions of $M_B(\gamma)$, along the hyperplane $x_n=0$, has a basis given by asymptotic expansions of certain holomorphic solutions of $M_B(\gamma)$, described by integral representations, as those considered by Adolphson in \cite[Sec.~2]{Adolphson}.

\begin{assumption}\label{assumption}
 We assume that the matrix $B$ satisfies  \begin{enumerate}
\item The points $b(d+1), \ldots, b(n-1)$ belong to the interior of $\Delta_\sigma$,
\item $b(n)$ is not  in $\Delta_\sigma $  and belongs to the open positive cone of $\sigma$.
\end{enumerate}
\end{assumption}

\begin{remark}
We notice that, under the above assumption, it follows from \cite{Adolphson} that, for any $\gamma\in \CC^d$,  the singular locus of the hypergeometric system $M_B (\gamma)$ is equal to
$\bigcup_{k\in \sigma} \{x_k =0\}\cup \{x_n =0\}$. Furthermore, by \cite{schulze-walther}, $M_B(\gamma)$ has a unique slope along the coordinate hyperplane $\{x_n=0\}$.
\end{remark}

In this Section we prove the following:
\begin{theorem}\label{Asymptotic-expansions-theorem}
In the above situation, let us assume that  $\gamma \in \CC^d$ and $\Re(B_\sigma^{-1}\gamma)<0$. Then all the Gevrey solutions of $M_B(\gamma)$ along the hyperplane $x_n=0$
can be described as linear combination of a fixed set of asymptotic expansions of integral solutions of type $I_{C_{p,\delta}}(\gamma; x)$.
Moreover, for each cycle $C=C_{p,\delta}$, there are meromorphic continuations with respect to $\gamma$ in $\CC^{d}$, of both $I_C(\gamma; x)$ and of the coefficients of the asymptotic expansion to the whole $\CC^{d}$. For any $\gamma\in \CC^d$ which is not a pole, the meromorphic continuation of $I_C(\gamma; x)$ has an asymptotic expansion whose coefficients are precisely the values at $\gamma$ of the meromorphic continuations of the coefficients of the asymptotic expansion of $I_C(\gamma; x)$.
\end{theorem}

In the next three subsections we are proving the analogous result for the reduced version $F_{p,\delta}$ of the hypergeometric integrals $I_{C_{p,\delta}}$; see Subsection \ref{reduced-version-F}. The transfer of the results to the integrals $I_{C_{p,\delta}}$ in the form of Theorem \ref{Asymptotic-expansions-theorem} is immediate.

\subsection{Existence of asymptotic expansions for the integrals.}
As we did in Subsection \ref{reduced-version-F}, let us consider  $\beta\in \CC^d$ and $A=(a(1),\dots,a(n))$ is a $d\times n$ matrix with rational coefficients and with $(a(1),\dots,a(d))=(e(1),\dots,e(d))$ the unit matrix. Let us denote by $|a|=a_1 +\cdots + a_d$ the sum of the coordinates of any vector $a\in \QQ^d$. Assumption \ref{assumption} takes the following form in this reduced presentation:
\begin{assumption}\label{assumption-2} The matrix $A$ satisfies:
\begin{enumerate}
\item For $j=d+1,\ldots, n-1$, the rational vector $a(j)$ is in the open positive quadrant in $\QQ^d$ and $|a(j)|<1.$
\item The rational vector $a(n)$ belongs to the open positive quadrant in $\QQ^d$ and $|a(n)|>1.$
\end{enumerate}
\end{assumption}

\begin{lemma}\label{remark-delta} For $\Re\beta<0$ and under Assumption \ref{assumption-2} one can find for each $y_{n,0} \in \CC^*$, and each $p\in \ZZ^d$ a value of the parameter $\delta\in \RR^d$ such that the integral $F_{p,\delta}(\beta; y)$ is absolutely convergent.
\end{lemma}
\begin{proof} Indeed one can choose $\delta=\delta (p, y_{n,0}) \in \RR^d$ with $|\delta_k|<1/2$ such that equation  $\eqref{sufficient}_n$ is satisfied. Such a $\delta$ exists because $|a(n)|>1$. Thus, the Lemma follows from Lemma $\ref{sufficientplus}$.
\end{proof}

\begin{remark}\label{sector-S-p-delta} Notice that  condition $\eqref{sufficient}_n$ involves a determination $\alpha_n$ of $\arg y_{n,0}$. Let us consider the open sector $S_{p,\delta}$ in $\CC^*$ around the direction $e^{i\alpha_n}$ which is the image of the connected component containing $\alpha_n$ of the set defined in $\wt{\CC^*}$ by $(\ref{sufficient})_n$ for the above mentioned $p,\delta$. \end{remark}

\begin{theorem} \label{ASY} If $\Re \beta_k <0$ for all $k=1,\ldots ,d$, then for any given $y_{n,0}\in \CC^*$ there is an asymptotic expansion with respect to the variable $y_n$ in the open sector $S_{p,\delta}$:
\[
F_{p,\delta}(\beta; y) \, \underset{y_n\to 0} {\sim} \, \sum_{m\in \NN}A_{p,\delta}(\beta; m, y')\frac{y_n^m}{m!}
\] where $y'=(y_{d+1},\ldots, y_{n-1})$ and $$A_{p,\delta}(\beta; m , y'):=\int_{D_{p,\delta}}t^{-\beta-\nn+ma(n)}\exp\left(t_1+\cdots+t_d + \sum_{j=d+1}^{n-1} y_j t^{a(j)}\right) dt.$$
\end{theorem}

\begin{proof}
We have to prove that for any integer $N > 0$ there exists $K_N =K_N (\beta , y')>0$ such that
$$\left|F_{p,\delta}(\beta; y) - \sum_{m=0}^{N-1}A_{p,\delta}(\beta; m , y')\frac{y_n^m}{m!}\right|\leq K_N  |y_n|^N$$ holds for every $y_n \in S_{p,\delta}$.

Let $$\Phi_N(z):= e^z-\sum_{m=0}^{N-1}\frac{z^m}{m!}$$ for $z\in \CC$. Then we have
\[
|\Phi_N(z)|\leq \dfrac{|z|^N}{N!} {\mbox { for all }} z {\mbox { such that }}  \Re z<0.
\]

Recall that by the assumption on $\delta$ we have $\Re (y_n t^{a(n)})<0$ when $t\in D_{p,\delta}$ since $y_n \in S_{p,\delta}$. Thus, we have

$$\left|F_{p,\delta}(\beta;y)-\sum_{m=0}^{N-1}A_{p,\delta}(\beta; m, y')\frac{y_n^m}{m!}\right|=
\left|y_n^N\right| \left|Q_{p,\delta}(\beta; y, N)\right|$$ where $$Q_{p,\delta}(\beta; y, N)=\int_{D_{p,\delta}} t^{a(n)N-\beta-\bf1}\exp{\left(t_1+\cdots+t_d + \sum_{j=d+1}^{n-1} y_j t^{a(j)}\right)}\frac{\Phi_N(y_n t^{a(n)})}{(y_n t^{a(n)})^N}{dt}.$$
Notice that the absolute value of the integrand is bounded by the function
$$\dfrac{1}{N!}\left| t^{a(n)N-\beta-\bf1}\exp{\left(t_1+\cdots+t_d + \sum_{j=d+1}^{n-1} y_j t^{a(j)}\right)}\right|$$ which is independent of $y_n$ and integrable over $D_{p,\delta}$ by Lemma \ref{sufficientplus} (that can be applied to the submatrix of $A$ defined by its first $n-1$ columns
because of Assumption \ref{assumption-2}). Notice that we use here that $\Re (\beta - a(n)N)<0$ for all $N>0$ since $a(n)$ does not have negative coordinates.

Thus, there exists $K_N =K_N (\beta, y')>0$, that can be locally bounded with respect to $(\beta, y')$, such that
$\left|Q_{p,\delta}(\beta; y, N)\right|\leq K_N$.
This finishes the proof.
\end{proof}

\begin{remark}\label{remark-first-condition}
Notice that if we assume conditions $(\ref{sufficient})_j$ for all $j\in \ol{\sigma}$, we don't need condition (1) in Assumption  \ref{assumption-2} in the proof of Theorem \ref{ASY}.
\end{remark}

\begin{remark}
The function $F_{p,\delta}(\beta ,y)$ is locally constant on $\delta$, which varies in a disjoint union of connected  open sets in $\RR^d$. It would be interesting to find an example, if any, with $\delta ,\delta'$ in different connected components, such that $F_{p,\delta}(\beta,y)\neq F_{p,\delta '}(\beta, y)$ for a fixed $p\in \ZZ^d$.
\end{remark}

\begin{remark}\label{AisF}
Notice that $A_{p,\delta}(\beta;m,y')=F_{p,\delta}(\beta-m a(n); y')$ for the submatrix of $A$ defined by its first $n-1$ columns.
\end{remark}

We extend Theorem \ref{ASY} to non negative values of $\Re \beta_k$ in Subsection \ref{analytic-continuation}.

\subsection{Analytic continuation with respect to $\beta$.}\label{analytic-continuation}
In this section we focus on the analytic dependency of  $F(\beta;y) = F_{p,\delta}(\beta;y)$ on $\beta$. Let us take $y_{n,0}\in \CC^*$ and $p\in \ZZ^d$.  We choose  $\delta$ as in Lemma \ref{remark-delta} and we omit $p, \delta$ in the remainder of this subsection. We assume now that $y$ belongs to $\CC^{n-d-1}\times S_{p,\delta}$, where the sector  $S_{p,\delta}$ is defined in Remark \ref{sector-S-p-delta}.

The integral $F(\beta; y)$ is a solution of the reduced GG-system (see \cite{GG99}):
\begin{equation}\label{2}
\beta_kF(\beta;y)=\sum_{\ell=d+1}^n a(\ell)_k \,y_{\ell}\, F(\beta-a(\ell);y) + F(\beta-e(k);y) \quad \text{ for } k=1,\ldots, d
\end{equation}
\begin{equation}\label{4}
F(\beta-a(\ell);y)=\frac{\p F}{\p y_{\ell}}(\beta;y) \quad \text{ for } \ell=d+1,\ldots,n.
\end{equation}

\begin{lemma}\label{meromorphic-continuation}
The function $F(\beta;y)$ admits a meromorphic continuation with respect to $\beta$, denoted by $\widetilde{F}(\beta;y)$,
with poles at most along the countable locally finite union of hyperplanes $$\mathcal{P}:=\bigcup_{k=1}^d\{\beta\in\CC^d|\; \beta_k\in \pi_k (\NN A)\}$$ where $\pi_k: \QQ^d \rightarrow \QQ$ denotes the projection to the $k$-th coordinate.
\end{lemma}
\begin{proof}
The initial domain of analyticity of $F(\beta;y)$ is defined by $\Re \beta_k <0$ for all $k=1,\ldots ,d$. Let us fix conditions $\Re \beta_k <0$ for $k=2,\ldots, d$ and extend the domain of analyticity in the coordinate $\beta_1$ using equation $(\ref{2})_1$ as follows. The functions $F(\beta-a(\ell);y)$ for $\ell=d+1,\ldots,n$ and $F(\beta-e(1);y)$ are analytic for $\Re \beta_1<\widetilde{a}_1:=\operatorname{min}_{\ell}\{a(\ell)_1 ,1\}$ and hence it follows from equation
$\eqref{2}_1$ that $F(\beta;y)$ is meromorphic in $\Re \beta_1 <\widetilde{a}_1$ with at most a pole in $\beta_1 =0$.

In the general inductive step for the variable $\beta_1$, we assume that $F(\beta;y)$ is meromorphic in the half--space $\Re \beta_1 <(q-1)\widetilde{a}_1$. Then, on the domain defined by $\Re \beta_1 < q\widetilde{a}_1$, the RHS of $(\ref{2})_1$ is meromorphic, with poles of type $\beta_1=c+1$, or $\beta_1=c+a(\ell)_1$, where $\beta_1=c$ runs over all the poles of $F(\beta;y)$. We obtain that $F(\beta;y)$ is also meromorphic in the same domain adding these new poles to those already found. Thus, by induction, we get that $F(\beta;y)$ is also meromorphic for $\beta_1\in \CC$ and $\Re \beta_k <0$ for $k=2,\ldots, d$ with poles at most along $\beta_1= \sum_{\ell=d+1}^n m_\ell a(\ell)_1 + m'$, for all $m_{d+1},\ldots, m_n, m'\in \NN$. By an analogous argument in $k=2,\ldots ,d$ we get the result.
\end{proof}

Notice that the equations \eqref{4} are then satisfied by  $\widetilde{F}(\beta;y)$ by analytic continuation on $U:=\CC^d\setminus \mathcal{P}$.

\begin{lemma}\label{asymptotic-expansion-tilde}
For any fixed $\beta\in U$, $\widetilde{F}(\beta;y)$ admits an asymptotic expansion along $y_n=0$ in $S_{p,\delta}$.
Furthermore, the coefficients $\wt{A} ({\beta}; m ,y')$ of this expansion are analytic with respect to $\beta\in U$. Hence they are analytic continuations of the coefficients $A (\beta; m, y')$ described in Theorem \ref{ASY}.
\end{lemma}

\begin{proof}
It follows from an induction starting from Theorem \ref{ASY} and parallel to the one used in the proof of Lemma \ref{meromorphic-continuation} that for any fixed $\beta\in U$, $\widetilde{F}(\beta;y)$ admits asymptotic expansions along $y_n=0$ in $S_{p,\delta}$.
By construction, these analytic continuations satisfy equation \eqref{2}, for any $\beta\in U$. This implies that the coefficients $\wt{A}({\beta}; m, y')$ of these expansions satisfy the following equations for $k=1,\ldots,d$:

\begin{equation}\label{coeff-induction}
\begin{split}
\beta_k\wt{A}({\beta}; m,y')=\wt{A}({\beta-e(k)}; m,y') + \sum_{\ell=d+1}^{n-1} a(\ell)_k \,y_\ell\, \wt{A}({\beta-a(\ell)}; m,y') \\
+ m a(n)_k \wt{A}({\beta-a(n)}; m-1,y').
\end{split}
\end{equation}

Again an induction like in Lemma \ref{meromorphic-continuation}, using (\ref{coeff-induction}) proves that $\wt{A}({\beta}; m,y')$ is analytic with respect to $\beta$ and $y'$, hence as a function of $\beta$ it is an analytic continuation to $U$ of $A (\beta; m, y')$.
\end{proof}

We have proved the following theorem which implies last sentence in Theorem \ref{Asymptotic-expansions-theorem} when we return to the integrals $I_C(\beta; y)$:

\begin{theorem} \label{ASY-tilde} There is an asymptotic expansion along $y_n=0,$ in an appropriate open sector $S_{p,\delta}$ around any half--line $\RR_{>0}\cdot y_{n,0}\subset\CC^*$:
\[
\widetilde{F}_{p,\delta}(\beta; y) \, \underset{y_n\to 0} {\sim} \, \sum_{m\in \NN}\widetilde{A}_{p,\delta}(\beta; m,y')\frac{y_n^m}{m!}
\] where $\widetilde{A}_{p,\delta}(\beta; m, y')$ is the analytic continuation of $A_{p,\delta}(\beta; m, y')$ to $\beta \in U$.
\end{theorem}

\subsection{Parametrizations}\label{Parametrizations}

We go on working with the reduced form of the integral described in  (\ref{reducedI-0}), (\ref{conv-I-to-F}) and (\ref{reducedI}), and  we study integrals of the form
\[
F_{p,\delta}(\beta; y) = \int_{D_{p,\delta}} t^{-\beta-\bf1}\exp{\left(t_1+\cdots+t_d+ \sum_{j=d+1}^{n} y_{j}t^{a(j)}\right)}dt.
\]

\begin{lemma}\label{Gamma-calculation} If $\Re \beta <0$, then
\[
A^0_{p,\delta}(\beta):=\int_{D_{p,\delta}}t^{-\beta-\nn}\exp(t_1+\cdots+t_d ) dt
=e^{\sqrt{-1}\,\pi\ideal{2p+{\bf1},-\beta}}\Gamma(-\beta)
\]
where $\Gamma(-\beta):=\prod_{k=1}^d\Gamma(-\beta_k).$
\end{lemma}
\begin{proof}
The integrand $t^{-\beta-\nn}\exp(t_1+\cdots+t_d ) dt$
is of rapid decay at infinity in the product of sectors from $D_{p,0}$ to $D_{p,\delta}$. Each of these sectors is given by the condition: $$\arg(t_k) \in [(1+\min\{0,\delta_k\}+2p_k)\pi, \, (1+\max\{0,\delta_k\}+2p_k)\pi].$$
Thus, we know by elementary considerations in one complex variable, that $A^0_{p,\delta}(\beta)$ does not depend on $\delta_k\in]-\frac{1}2,\frac{1}2[$ and, in particular, $A^0_{p,0}(\beta)=A^0_{p,\delta}(\beta)$.

We parametrize
$D_{p,0}$  by
$t_k=\rho_k e^{\sqrt{-1}\,\pi (2p_k+1)}=-\rho_k$ with $\rho_k \in \, ]0,+\infty)$, and the result follows directly from the expression that we obtain:
$$A^0_{p,0}(\beta)=\int_{]0,+\infty)^d}\exp(\sqrt{-1}\,\pi\ideal{2p+\nn,-\beta})\rho^{-\beta-\nn}\exp(-\rho_1+\cdots-\rho_d ) d\rho.$$
\end{proof}

Since $A^0_{p,\delta}(\beta)$ does not depend on $\delta$, from now on we drop $\delta$ and set $A^0_p(\beta) := A^0_{p,0}(\beta)=A^0_{p,\delta}(\beta)$.

The dependency on $\delta$ of $F_{p,\delta}(\beta; y)$ must be kept, even if the integral is locally constant with respect to $\delta$, because the argument by homotopy to reduce $ \delta$ to zero in $A^0_{p,\delta}$, does not work, due to the presence in the argument of the exponential of the term $y_nt^{a(n)}$ which prevents the rapid decay property from being kept along the homotopy.

Let us now make the analytic continuation of the coefficients of the asymptotic expansion described in Theorem \ref{ASY} more precise by developing them with respect to $y'$.

\begin{lemma}\label{conv-coeff} The coefficients of the asymptotic expansion described in Theorem \ref{ASY} are analytic functions of the variables $y'$ with the following power series development:
\begin{align}\label{Taylor-series}
 A_{p,\delta}(\beta; m , y')= \sum_{{\bf m}'\in\NN^{n-d-1}}A^0_p \left(\beta-ma(n)-\sum_{j=d+1}^{n-1}m_ja(j)\right)\frac{y'^{\,{\bf m}'}}{{\bf m}' !}.
\end{align}
Furthermore, this expansion is still valid for the meromorphic continuation of $A_{p,\delta}(\beta; m , y')$ found in Lemma \ref{asymptotic-expansion-tilde} and the meromorphic continuation of $A^0_p(\beta)$ deduced from Lemma \ref{Gamma-calculation}.
\end{lemma}
\begin{proof}
Recall that, when $\Re\beta<0$ the coefficient we consider has the form $$\displaystyle A_{p,\delta}(\beta; m , y')=\int_{D_{p,\delta}}\varphi(\beta,y';t) dt$$ with
$$\varphi(\beta,y';t)=t^{-\beta-\nn+ma(n)}\exp\left(t_1+\cdots+t_d + \sum_{j=d+1}^{n-1} y_j t^{a(j)}\right).$$
We set $|t_k|=\rho_k$ for $k=1,\dots,d$ and we parametrize $D_{p,\delta}$ by $\rho\in\RR_{>0}^d$. We fix a polydisc $Q=\{y'\mid |y_j|<R_j, j=d+1,\dots,n-1\}\subset\CC^{n-1-d}$.

By the same argument as in the proof of Lemma \ref{sufficientplus} and inequality \eqref{dominant-exponent}, the integrand $\varphi(\beta,y';t)dt$ is dominated, via the parametrization $t_k=e^{(1+\delta_k+2p_k)\sqrt{-1}\,\pi}\rho_k$ and up to a constant factor, by $$\rho^{-\Re\beta+ma(n)-\nn}\exp(C-c(\rho_1+\dots+\rho_d))d\rho$$ for some constants $C,c\in \RR_{>0}$. These constants depend only on $Q$ but not on $y'\in Q$ by Remark \ref{uniform-constants} applied to $A_{p,\delta}(\beta; m , y')$ instead of $F_{p,\delta}(\beta; y)$.

The function $\varphi(\beta,y';t)$ is holomorphic with respect to $y'\in\CC^{n-d-1}$.

For each $j=d+1,\dots,n-1$, the integral  $$\displaystyle\int_{D_{p,\delta}}\frac{\p\varphi(\beta,y';t)}{\p y_j} dt$$ has an expression similar to the one for $A_{p,\delta}(\beta; m , y')$, with $\beta $ replaced by $\beta-a(j)$. By the same argument as for $\varphi$, the integrand $\displaystyle\frac{\p\varphi(\beta,y';t)}{\p y_j} dt$ is dominated, up to a constant factor, by $$\rho^{-\Re\beta+ma(n)+a(j)-\nn}\exp(C_j-c_j(\rho_1+\dots+\rho_d))d\rho$$ for some constants $C_j,c_j\in \RR_{>0}$, independent of $y'$ in the polydisc $Q$.

By Lebesgue dominated convergence theorem for integrals, this proves that $A_{p,\delta}(\beta; m, y')$ is holomorphic with respect to $y'$  and that
\[
\frac{\p A_{p,\delta}(\beta; m , y')}{\p y_j}=\int_{D_{p,\delta}}\frac{\p\varphi(\beta,y';t)}{\p y_j} dt
\] for all $j=d+1,\ldots,n-1$.
If we iterate the argument we obtain an expression of the partial derivatives of $A_{p,\delta}$, up to any order ${\bf m}'=(m_{d+1},\dots,m_{n-1})$:
\[
\frac{\p^{|{\bf m}'|} A_{p,\delta}(\beta; m , y')}{\p^{m_{d+1}}y_{d+1}\dots\p^{m_{n-1}} y_{n-1}}=\int_{D_{p,\delta}}\frac{\p^{|{\bf m}'|}\varphi(\beta,y';t)}{\p^{m_{d+1}}y_{d+1}\dots\p^{m_{n-1}} y_{n-1}} dt.
\]
Setting $y'=0$ in this last expression gives the coefficients of the Taylor expansion of $A_{p,\delta}(\beta; m , y')$ with respect to $y'$ at the origin. This proves the equality (\ref{Taylor-series}) when $\Re\beta<0$.

The last claim of this lemma follows from the explicit calculation in Lemma \ref{Gamma-calculation} from which we see that the coefficient of $\frac{y'^{\,{\bf m}'}}{{\bf m}'!}$ is equal to
\begin{align*}A^0_p\left(\beta-ma(n)- \sum_{j=d+1}^{n-1}m_ja(j)\right)=&\\e^{\sqrt{-1}\,\pi\ideal{2p+{\bf1},-\beta + ma(n)+  \sum_{j=d+1}^{n-1}m_ja(j)}}&\Gamma\left(-\beta + m a(n)+\sum_{j=d+1}^{n-1}m_ja(j)\right).
\end{align*}

By the standard properties of the $\Gamma$--function, this coefficient admits a meromorphic continuation with respect to $\beta$, with poles along a subset of $\mathcal{P}$ defined in Lemma \ref{meromorphic-continuation}.

When $\beta\in\CC^d\setminus\mathcal{P}$ the right hand side of (\ref{Taylor-series}) is still defined and yields a convergent power series defined for all $y'\in \CC^{n-d-1}$ because of the conditions $|a(j)|<1$ for $j=d+1,\ldots,n-1$.

Therefore it is an analytic continuation of the power series defined for $\Re\beta<0$. The equality \eqref{Taylor-series} follows everywhere in $\CC^d\setminus\mathcal{P}$ with the previously defined meromorphic continuation of $A_{p,\delta}(\beta; m , y')$ on the LHS.
\end{proof}

\begin{remark}
Notice that as a consequence of Lemma \ref{conv-coeff} the function $A_{p,\delta}(\beta; m , y')$ does not depend on $\delta$.
\end{remark}

\subsection{Space of asymptotic expansions and Gevrey series.}\label{Asymptotic-Gevrey}
In this section we finish the proof of Theorem \ref{Asymptotic-expansions-theorem}.

For any ${\bf k}\in \NN^{n-d}$, let us set $$\Lambda_{\mathbf{k}} :=\{\mathbf{k}+\mathbf{m}=(k_{d+1}+m_{d+1},\ldots, k_n+ m_n )\in \N^{n-d}: \; A_{\ol{\sigma}}\mathbf{m} \in \Z^d \}$$ and define
$$S_{\mathbf{k}}(\beta; y) := \sum_{\mathbf{k}+\mathbf{m}\in \Lambda_{\mathbf{k}}} e^{|A_{\ol{\sigma}}(\mathbf{k}+\mathbf{m})|\pi \sqrt{-1}} \Gamma\left(-\beta+ A_{\ol{\sigma}}(\mathbf{k}+\mathbf{m})) \right)\frac{y^{\mathbf{k}+\mathbf{m}}}{(\mathbf{k}+\mathbf{m})!}.$$
Notice that the coefficients of the series $S_{\mathbf{k}}$ are meromorphic with respect to $\beta\in \CC^d$ with at most simple poles along $\mathcal{P}$. In particular, if $\beta\notin \mathcal{P}$ all these series are well defined nonzero power series with support equal to $\Lambda_{\mathbf{k}}$ since the Gamma function does not have any zero. It can be proved by using standard estimates of Gamma functions that these series are Gevrey along $y_n=0$ with Gevrey index $|a(n)|>1$.

Let $\Omega\subseteq\NN^{n-d}$ be a set of cardinality $[\ZZ A : \ZZ A_{\sigma}]$ such that $$\{ A_{\ol{\sigma}}\mathbf{k}+ \ZZ A_{\sigma}:\; \mathbf{k} \in \Omega\}=\ZZ A / \ZZ A_{\sigma}=\ZZ A / \ZZ ^d.$$ We notice that the existence of such $\Omega\subseteq\NN^{n-d}$ follows from \cite[Lemma~3.2]{F}. It is clear that $\mathcal{G}=\{S_{\mathbf{k}}(\beta; y):\; \mathbf{k} \in \Omega\}$ is a linearly independent set because the series $S_{\mathbf{k}}$ have pairwise disjoint supports $\Lambda_{\mathbf{k}}$.

Using Theorem \ref{ASY}, Lemma \ref{Gamma-calculation} and Lemma \ref{conv-coeff}, we have
\begin{align*}\label{connection--formula}
F_{p,\delta}(\beta; y)&\underset{y_n\to 0}{\sim}\sum_{q_n \in \NN} A_{p,\delta}(\beta; q_n,y') \frac{y_n^{q_n}}{q_n!}\\
=& \sum_{\mathbf{q}\in \NN^{n-d}} A^0_p(\beta-A_{\ol{\sigma}}\mathbf{q})\frac{y^{\mathbf{q}}}{\mathbf{q}!} =
\sum_{\mathbf{q}\in \NN^{n-d}} e^{\sqrt{-1}\,\pi\ideal{\nn+2p,-\beta+A_{\ol{\sigma}}\mathbf{q}}} \Gamma\left(-\beta+A_{\ol{\sigma}}\mathbf{q}\right)\frac{y^{\mathbf{q}}}{\mathbf{q}!}\\
=&
e^{\sqrt{-1}\,\pi\ideal{\nn+2p,-\beta}}\sum_{\mathbf{k}\in \Omega}
\sum_{\mathbf{k}+\mathbf{m}\in \Lambda_{\mathbf{k}}}
e^{\sqrt{-1}\,\pi\ideal{\nn+2p,A_{\ol{\sigma}}(\mathbf{k}+\mathbf{m})}}
\Gamma\left(-\beta+A_{\ol{\sigma}}(\mathbf{k}+\mathbf{m})\right)
\frac{y^{\mathbf{k}+\mathbf{m}}}{(\mathbf{k}+\mathbf{m})!}\\
 =& e^{\sqrt{-1}\,\pi\ideal{\nn+2p,-\beta}}\sum_{\mathbf{k}\in \Omega}
 e^{\sqrt{-1}\,\pi\ideal{\nn+2p,A_{\ol{\sigma}}\mathbf{k}}}
 \sum_{\mathbf{k}+\mathbf{m}\in \Lambda_{\mathbf{k}}}
 e^{\sqrt{-1}\,\pi\ideal{\nn,A_{\ol{\sigma}}\mathbf{m}}}
 \Gamma\left(-\beta+A_{\ol{\sigma}}(\mathbf{k}+\mathbf{m})\right)
 \frac{y^{\mathbf{k}+\mathbf{m}}}{(\mathbf{k}+\mathbf{m})!}\\
=& e^{\sqrt{-1}\,\pi\ideal{\nn+2p,-\beta}}\sum_{\mathbf{k}\in \Omega}
e^{\sqrt{-1}\,\pi\ideal{2p,A_{\ol{\sigma}}\mathbf{k}}} S_{\mathbf{k}}(\beta; y).
\end{align*}

Notice that previous power series is formal with respect to $y_n$, with convergent coefficients. More precisely, it is a Gevrey series along $y_n=0$ with Gevrey index $|a(n)|>1$.
We notice also that $\Re \beta <0$ implies that $\Re (\beta -A_{\overline{\sigma}}\mathbf{q})<0$ for all $\mathbf{q}\in \NN^{n-d}$, by using Assumption (4.4), which guarantees the convergence of all the integrals involved in  Subsection \ref{Parametrizations}.
By the last claim in Lemma \ref{conv-coeff}, this calculation is valid everywhere in the domain of analytic continuation $\CC^d\setminus\mathcal{P}$, since the argument applies also to the coefficients of the series $S_{\bf k}(\beta; y).$

The matrix of coefficients of the series $S_{\mathbf{k}}(\beta; y)$ in the asymptotic expansions of the functions
$$e^{\sqrt{-1}\,\pi \ideal{2p+{\bf1},\beta}} F_{p,\delta}(\beta; y)$$ is $( e^{\sqrt{-1}\,\pi\ideal{2p,A_{\ol{\sigma}}\mathbf{k}}})_{\mathbf{k},p}$ where $\mathbf{k}$ varies in $\Omega$. If $p$ varies in an appropriate set of $[\ZZ A: \ZZ^d]$ elements, this matrix is square invertible. Indeed, we have $\ZZ A /\ZZ A_{\sigma}=\ZZ A /\ZZ^d\simeq \ZZ^d/ \ZZ M$, where $M$ is the matrix of coordinates of the canonical basis of $\ZZ^d$ with respect to a basis
of $\ZZ A /\ZZ^d$. Thus, the matrix $( e^{\sqrt{-1}\,\pi\ideal{2p,A_{\ol{\sigma}}\mathbf{k}}})_{\mathbf{k},p}$ is invertible by \cite[Proposition~6.3]{Matsubara-Heo}, if $p$ runs in a set of representatives of the quotient $\ZZ^d/ \ZZ~^tM$.

In particular, if $\beta\notin\mathcal{P}$ the set of holomorphic functions $F_{p,\delta}(\beta; y)$, where $p$ varies in this set of representatives, is also a linearly independent set and any Gevrey series along $y_n=0$ in the space generated by the series $\{S_{\mathbf{k}}(\beta; y):\; \mathbf{k}\in\Omega\}$ is an asymptotic expansion of a linear combination of the integrals $F_{p,\delta}(\beta; y)$.

Now if we start from the matrix $B$ in Section \ref{intro} and we apply the above results with the matrix $A=B_{\sigma}^{-1}B=(I,B_{\sigma}^{-1}B_{\ol\sigma})$, and the parameter $\beta=B_{\sigma}^{-1}\gamma$, we obtain a similar statement for the integrals $I_C(\gamma,x)$ using (\ref{conv-I-to-F}) and (\ref{reducedI}) if we set $y_j=x_j x_{\sigma}^{-a(j)}$ for all $j  = d+1,\ldots,n$, or $y=x_{\ol\sigma}x_\sigma^{-B_{\sigma}^{-1}B_{\ol\sigma}}$. Moreover, in this case $M$ can be chosen to be $B_{\sigma}$. In particular, we get that $x_\sigma^{B_{\sigma}^{-1}\gamma}\cdot\mathcal{G}$
is a linearly independent set of Gevrey series solutions of $M_B (\gamma)$ along $x_n=0$ with Gevrey index $|a(n)|=|B_{\sigma}^{-1}b(n)|>1$ if $\beta\notin\mathcal{P}$.

It is enough to prove that the dimension of the space of Gevrey series solutions of $M_B (\gamma)$ along $x_n=0$ is at most equal to $|\Omega|=[\ZZ^d:\ZZ B_\sigma]$ when $\Re \beta<0$. To this end, notice first that, if $$f=\sum_{m=0}^\infty f_m (x_1,\ldots ,x_{n-1})x_n^m$$ is a Gevrey series belonging to this space, then the initial part of $f$ with respect to the weight vector $w=(0,\ldots,0,1)\in\RR^n$ has the form $\operatorname{in}_w (f)=f_m (x_1,\ldots, x_{n-1})x_n^m$ for some $m\geq 0$ and it is hence a holomorphic function. Thus, by the same argument as in the proof of \cite[Th.~2.5.5]{SST}, it is a (holomorphic) solution of $\operatorname{in}_{(-w,w)}(H_B(\gamma))$. This last ideal is the initial ideal with respect to $w$ of the hypergeometric ideal associated with $(B,\gamma)$ (see \cite[p.~4]{SST}). In particular, the dimension of the space of Gevrey solutions is at most equal to the rank of $\operatorname{in}_{(-w,w)}(H_B(\gamma))$, because one can choose a basis of Gevrey solutions of $M_B(\gamma)$ such that their initial parts are also linearly independent (see \cite[Proposition~2.5.7]{SST}).

On the other hand, by using \cite[Lemma~2.1.6]{SST} for $(u,v)=(\mathbf{0},\mathbf{1})$ and $(u',v')=(-w,w)$, we have that the characteristic ideal of $\operatorname{in}_{(-w,w)}(H_B(\gamma))$ is $$\operatorname{in}_{(\mathbf{0},\mathbf{1})}(\operatorname{in}_{(-w,w)}(H_B(\gamma)))=\operatorname{in}_{L}(H_B(\gamma))$$ for $L=(-w,w)+\epsilon (\mathbf{0},\mathbf{1})$ with $\epsilon>0$ small enough.

Thus, by \cite[Th.~4.21, Rk.~4.23 and Th.~4.28]{schulze-walther} for $L=(-w,w)+\epsilon(\mathbf{0},\mathbf{1})$ and Assumption \ref{assumption}, we have that the holonomic rank of $\operatorname{in}_{(-w,w)}(H_B(\gamma))$ equals $|\Omega|$ if $\gamma$ is not \emph{rank--jumping} for $B$ (that is, if $\rank (M_B(\gamma))=d!\operatorname{Vol}(\Delta_B)$), a condition that is weaker than $\Re\beta=\Re (B_{\sigma}^{-1}\gamma)<0$ by \cite[Th.~5.15]{Adolphson} (see also \cite[Cor.~4.5.3]{SST}). This finishes the proof of Theorem \ref{Asymptotic-expansions-theorem}.

\begin{remark}
If the first condition in Assumption \ref{assumption-2} is not satisfied, previous argument is no more valid because there will be at least one $a(j)$, for $d+1 \leq j \leq n-1$, with at least one negative entry. By Lemma \ref{sufficientplus}, $F_{p,\delta}(\beta ; y)$ is absolutely convergent if we require that the monomials $y_{\ell}t^{a(\ell)}$ have negative real part for $\ell\in \tau_A$. Thus, Theorem \ref{ASY} remains valid in this case if we add the condition $\Re (y_{\ell}t^{a(\ell)})<0$ for all $\ell\in \tau_A$.
However, in this case the coefficients $A_{p,\delta}(\beta-A_{\overline{\sigma}}\mathbf{m})$ that appear in the previous proof are not convergent anymore if $m_\ell$ is big enough for $\ell \in \tau_A$ such that $a(\ell)$ has some negative entries. This happens because in this case $\Re (\beta -A_{\overline{\sigma}}\mathbf{m})$ will have some positive entries for some  $\mathbf{m}$.
\end{remark}

\begin{remark}
Notice that the proof of Theorem \ref{Asymptotic-expansions-theorem} shows that the constructed  set of Gevrey series solutions $x_\sigma^{B_{\sigma}^{-1}\gamma}\cdot\mathcal{G}$ is still a basis of the space of Gevrey solutions of $M_B (\gamma)$ along $x_n=0$ when $\gamma$ is not rank--jumping and $\beta=B_{\sigma}^{-1}\gamma\notin\mathcal{P}$, where  ${\mathcal {P}}$ is defined in Lemma \ref{meromorphic-continuation}. We don't know if under Assumption \ref{assumption} the condition of $\gamma$ being rank--jumping implies $\beta\in \mathcal{P}$. However, it is true that if $\gamma$ is rank--jumping then it is  \emph{semi--resonant} \cite{Adolphson}. In particular, under Assumption \ref{assumption}, $\gamma$ is semi--resonant for $B$ if and only if $\beta\notin\mathcal{P}':=\cup_{k=1}^d\{\beta\in\CC^d|\; \beta_k\in \pi_k (\ZZ A \cap \RR_{\geq 0}^d)\}$ where $\pi_k$ is the projection to the $k$-th coordinate. Notice also that $\mathcal{P}\subseteq\mathcal{P}'$.
\end{remark}

\begin{remark}
Notice that, using Euler's reflection formula, it can be easily shown that, for all $\mathbf{k}\in\Omega$ and when $\beta$ is generic enough,
$$S_{\mathbf{k}}(B_{\sigma}^{-1}\gamma; \, x_{\ol\sigma}x_\sigma^{-B_{\sigma}^{-1}B_{\ol\sigma}})=\dfrac{\pi^d e^{\sqrt{-1}\,\pi |A_{\overline{\sigma}}\mathbf{k}|}}{\sin (\pi (-\beta +A_{\overline{\sigma}}\mathbf{k} ))}\cdot \varphi_{v^{\mathbf{k}}} .$$
The series $\varphi_{v^{\mathbf{k}}}$ are used in \cite[Sec.~3]{F} in order to construct Gevrey series solutions for the hypergeometric system $M_B (\gamma)$. The genericity condition here means that $\beta\notin\mathcal{P}$ and that $\beta-A_{\overline{\sigma}}\mathbf{k}$ does not have integer coordinates for all  $\mathbf{k}\in\Omega$.
\end{remark}

\section{Integrals over rapid decay cycles.}\label{Limit-of-integrals-rapid-decay}

The goal of this Section is to prove that when $\Re\beta<0$ is sufficiently general, the integrals
studied in Theorem \ref{Asymptotic-expansions-theorem}, are in fact integrals over rapid decay cycles in the sense of \cite{Hien09}. These integrals are defined without the condition $\Re\beta<0$ and are still solutions of our GKZ system when $\Re\beta_k\geq 0$ for some $k$. By meromorphic continuation proved in Theorem \ref{ASY-tilde} they admit asymptotic expansions as Gevrey series solution for all $\beta$ sufficiently general in $\CC^d$.

\subsection{Description of rapid decay cycles.}\label{Hiensufficient0}
In this section we first briefly recall the theory of rapid decay homology by M. Hien in \cite[Sec.~5.1]{Hien09} and give a sufficient condition to detect a cycle for this homology.

Let $U$ be a complex quasi-projective variety over $\CC$ of dimension $d$. Let $h\in \cO(U)$
and let $X$ be a smooth projective compactification of $U$, such that $D = X\setminus U$ is a normal crossing divisor, and $h$ extends to a map $h : X\too \PP^1$.

Let us denote by $\pi: \widetilde{X}(D) \too X^{an}$ the real oriented blow-up along $D$ as defined in \cite[8.2]{Sabbah}. The space  $\widetilde{X}:=\widetilde{X}(D)$ can be embedded  into a real Euclidian space as a semi--analytic subset, and $h$ induces a map $\wt{h} : \widetilde{X}\too \PP^1.$

Let us describe the morphism $\pi$, locally at $p\in D$ with local coordinates $t_1,\dots, t_d$ such that $p=0$ and $D = \{t_1\cdots t_k = 0\},$
$$\begin{array}{ccc}
\pi :\; ([0,\epsilon)\times S^1)^k\times B(0,\epsilon)^{d-k}& \too & \CC^d\\
((r_j,e^{\sqrt{-1}\,\theta_j})_{j=1}^k,t') &\mapsto &(r_1\cdot e^{\sqrt{-1}\,\theta_1},\dots,r_k\cdot e^{\sqrt{-1}\,\theta_k},t')
\end{array}$$ where
$t'=(t_{k+1},\dots,t_d)$ and $\epsilon >0$ is a small real number.

Consider a regular flat algebraic connection $\nabla : \cE \too \cE \otimes \Omega_U$,
restriction to $U$ of  a regular meromorphic connection $(\cE_X(*D),\nabla)$ on $X$, with $\cE_X$ a lattice of this connection.

On the oriented blow-up $\widetilde{X}$, we consider the sheaf $\cA_{\wt{X}}^{<D}$ of holomorphic functions which are flat along $\wt{D}:=\pi^{-1}(D)$.

A section of $\cA_{\wt{X}}^{<D}$ on an open set $\wt{V}\subseteq \wt{X}$ is a holomorphic function $u$ on $\wt{V}\setminus \wt{D}$ such that, for any compact set $K \subset \widetilde{V}$,
and all $N=(N_1,\ldots, N_k)\in\NN^k$, there exists a constant $C_{K,N}>0$ satisfying:
\begin{equation}\label{flatnesscondition}
|u(t)| \leq C_{K,N} |t_1 |^{N_1}\cdots|t_k |^{N_k},\quad \forall  t=(t_1,\ldots,t_d)\in K\setminus \wt{D}
\end{equation}
in terms of local coordinates as above such that locally $D =\{t_1\cdots t_k = 0\}$.

The twisted connection $\nabla_h=\nabla - dh\wedge=\exp(h)\circ\nabla\circ\exp(-h)$ on $U$ extends to
a morphism of sheaves over $\wt{X}$ $$\cE\otimes \cA_{\wt{X}}^{<D} \longrightarrow \cE\otimes \cA_{\wt{X}}^{<D} \otimes_{\pi^{-1} {{\mathcal O_{X^{an}}}}} \pi^{-1} \Omega^1_{X^{an}}(*D).$$
The kernel of this extension is denoted by $\cS^{<D}$. The restriction of this kernel to $U$ is the set of horizontal sections of $\nabla_h$ and it is equal to $\cS\otimes \exp(h)$
where $\cS$ is the local system of horizontal sections of $\nabla$. Since the coefficients of a section of
$\cS$, on a basis of $\cE$ at a point $P\in D$, have at most a polynomial growth, the germ of $\cS^{<D}$
at a point $\wt{P}\in\wt{D}$ is non zero if and only if $\exp(\wt{h})$ satisfies condition (\ref{flatnesscondition}) at $\wt{P}$. At such a point it is equal to the germ of $\pi^{-1}(\cS)\otimes \exp(\wt{h})$.

The sheaf of rapid decay chains \cite[Sec.~5.1]{Hien09} is obtained from the {\it sheaf} $\cC^{-p}_{\wt{X},\wt{D}}$ {\it of relative chains mod $\wt{D}$} by tensoring it with $\cS^{<D}$ :
\[
 \cC^{r.d., - p}_{\wt{X}} :=\cC^{-p}_{\wt{X},\wt{D}}\otimes_\C\cS^{<D}.
\]

Let $j: U\hookrightarrow \wt{X}$ be the inclusion map. The sheaf $\cS^{<D}$ is a subsheaf
of $j_*(\cS\otimes\exp(h))$, which is isomorphic to $j_*(\cS)$
through the multiplication by $\exp(-h)$. Therefore $\cC^{r.d., - p}_{\wt{X}}$
is a subsheaf of $\cC^{-p}_{\wt{X},\wt{D}}\otimes_\C j_*(\cS\otimes\exp(h))$,
and in the next lemma we determine its image in
$\cC^{-p}_{\wt{X},\wt{D}}\otimes_\C j_*(\cS)$. In all what follows we
identify $\cC^{r.d., - p}_{\wt{X}}$ with this image. This convention is the
most appropriate for the expression of integrals.
\begin{lemma}\label{Hiensufficient}
Let $\varSigma$ be a semi--algebraic set in $U$ such that $\Re h$ tends
to $-\infty$ on $\varSigma$ with a controlled argument for $h$, i.e.,
there exists  $\delta\in  \left]0,\frac{\pi}2\right[$ such that for all $R>0$
there exists a compact set $K\subset U$, such that, for all
$t\in\varSigma\setminus K$ we have:
\[
\Re(h(t))<-R \text{ and } Arg (h(t))\in  \left]\pi-\delta,\pi+\delta\right[.
\]
Then the closure $\ol{\varSigma}$ in $\wt{X}$ is a compact semi--algebraic set.
Moreover, let ${\mathfrak T}$ be any finite triangulation of $\ol{\varSigma}$,
and let $\varUpsilon:=\sum_\Delta \Delta\otimes \varsigma_\Delta$ be a section
of $\cC^{-d}_{\wt{X},\wt{D}}\otimes_\C j_*(\cS)$.
This is a finite sum where $\Delta$ runs, possibly with repetitions, over all the $d$-simplices of ${\mathfrak T}$ which are not included in $\wt{D}$ and $\varsigma_\Delta$ is a section of $\cS$ over $j^{-1}(\Delta)\cap U$. Then $\varUpsilon$ is a rapid decay
chain, whose support is contained in $\ol{\varSigma}$.
\end{lemma}
\begin{proof}
The divisor $D$ is the union of the components $(D_i)_{i\in I}$ of $h^{-1}(\infty)$ and of other components $(D'_j)_{j\in J}$, such that on each $D'_j$ the restriction of $h$ is surjective on $\PP^1$ or takes a finite constant value. From the fact that $|h(t)|$ tends to $+\infty$ on the support  $\supp(\varUpsilon):=\bigcup\{\Delta\mid\varsigma_\Delta\ne0\}$ of $\varUpsilon$, we deduce that $\ol{\supp({\varUpsilon})}\cap D\subset h^{-1}(\infty)=\bigcup_{i\in I}D_i$, where $\ol{\supp({\varUpsilon})}$ denotes the closure of $\pi(\supp({\varUpsilon}))$ in $X$. Furthermore the closure of $\supp(\varUpsilon)$ in $\wt{X}$ meets $\wt{D}$ only at points such that $\arg \tilde{h}\in [\pi-\delta,\pi+\delta]$.

Let $\tilde{P}\in \wt{D}$ be such a point and let $P=\pi(\tilde{P})$. We choose local coordinates $(v_1,\dots,v_d)$ centered at $P$ such that a local equation of $D$ is $v_1\dots v_k=0$. The local expression of $h$ is
\[
h(v)=\frac{w(v_1,\dots,v_d)}{v_1^{m_1}\dots v_k^{m_k}}
\]
with all $m_k>0$ and $w(v)$ a unit since there are no points of indeterminacy. In a small enough neighbourhood of $P$, $h(v)=-\exp(i\delta(v))\vert h(v)\vert$, with
$\cos\delta(v)\geq \cos\delta>0$ and $|w(v_1,\dots,v_d)|\geq R'$, for some $R'>0$. Finally, around $P$ we obtain the expected rapid decay condition because:
\[
|\exp(h(v))|=\exp(\Re h(v))\leq \exp\left(-\frac{R'\cos\delta}{|v_1|^{m_1}\dots |v_k|^{m_k}}\right).
\]
\end{proof}
In order to treat integrals $I_C(\beta;x)$ as in the introduction, we consider the connection $(\cO_U,\nabla_\beta)$ on $U=\left(\CC^*\right)^d$
with the differential $\nabla_\beta=d+(\beta+1)\frac{dt}{t}\wedge$
and its meromorphic extension $(\cO_X(*D),\nabla_\beta)$ to $X$. It contains a lattice isomorphic to $\cO_X$, and the local system of horizontal sections over $U$ is $\cS_\beta=\CC\cdot t^{-\beta-1}$. We set $h(t)=\sum_{\ell=1}^n x_\ell t^{a(\ell)}$
and we intend to apply Lemma \ref{Hiensufficient} for a fixed value of $x$.
For that purpose, we have to use a cycle different from the cycles $C_{p,\delta}$, considered in Section \ref{intro}, since the support of $C_{p,\delta}$ always have the origin of $\CC^d$ in its closure, and when $t$ tends to $0$ along $C_{p,\delta}$, $h$ does not tend to $+\infty$. This cycle is described in detail in the next section \ref{the rd cycle}. It is is still a Borel-Moore cycle on the universal covering $(\wt{\CC^*})^d$, whose projection $\Sigma$
on $(\CC^*)^d$ is semi--algebraic. There is a triangulation $\mathfrak{T}$ of its closure $\ol{\Sigma}$ in $\wt{X}$ and a set
of $d$-simplices $\Delta\in\mathfrak{T}$ not contained in $\wt{D}$, such that $C$ is obtained by taking their restriction
to $(\CC^*)^d$ and an appropriate lifting to the universal covering
$(\wt{\CC^*})^d$. These liftings induce determinations
$\varsigma_\Delta=\left(t^{-\beta-1}\right)_\Delta$ of $t^{-\beta-1}$ and we
identify $C$ with the twisted chain : $\sum \Delta\otimes \varsigma_\Delta$.
The formula in \cite[page~23]{AoK}, can be directly adapted to the
irregular case :
\[
\int_{\Delta\otimes \varsigma_\Delta}t^{-\beta-1}e^{h(x,t)}dt=\int_\Delta \left(t^{-\beta-1}\right)_\Delta e^{h(x,t)}dt
\]
and the construction above shows that the integral $I_C(\beta;x)$ along $C$ is the integral along this twisted cycle.
\begin{corollary}\label{Hiensufficient2}
Let us assume that in the above situation $h(x,t)$ satisfies the condition of rapid decay and controlled argument in Lemma \ref{Hiensufficient}. Then the cycle $\varUpsilon=\sum \Delta\otimes \varsigma_\Delta$ associated with $C$ is a rapid decay cycle, and the integral
$I_C(\beta;x)$ along this cycle is convergent.
\end{corollary}
\begin{proof}
Only the last assertion requires a proof. Consider again a point $P\in D$, with coordinates  $(v_1,\dots,v_d)$ as in the last argument for Lemma \ref{Hiensufficient}. Since $t$ is algebraic,
$t^{-\beta-1}$ has at most a polynomial growth around $P$, with respect to $\displaystyle\frac1{v_1\cdots v_k}$. Therefore $\displaystyle t^{-\beta-1}e^{h(x,t)} $ is locally bounded by an expression of the form $\displaystyle\frac1{|v_1\cdots v_k|^m}\exp\left(-\frac{R'\cos\delta}{|v_1|^{m_1}\cdots |v_k|^{m_k}}\right)$ for some integer $m>0$. This yields a convergent integral on $\cU_P\cap (\CC^*)^d$ for some closed neighbourhood $\cU_P$ of $P$. Since $\ol{\Sigma}\cap D$ can be covered by a finite number of such $\cU_P$, the integral is indeed convergent.
\end{proof}
From now on we will identify a cycle on the universal covering and the corresponding twisted cycle and denote it by the same symbol.

\subsection{Realization of solutions by integrals over rapid decay cycles.}\label{the rd cycle}
We state and prove here the main result of this section.
Let us recall that we define $q_k$, for $k=1,\ldots,d$, as the lowest common denominator of the $k$-th row of $A$, that is as the smallest integer such that $q_ka(\ell)_k\in\NN$, for $\ell=d+1,\dots,n$ (cf. Remark \ref{finite_covering}).

\begin{theorem}\label{limit-rapid-decay}
There is a rapid decay cycle $\widetilde{D}_{p,\delta}$ such that the integral
$$\int_{\widetilde{D}_{p,\delta}} t^{-\beta-{\bf 1}}\exp\left(t_1 + \cdots + t_d +\sum_{j=d+1}^n y_{j} t^{a(j)}\right) dt $$ is equal to $F_{p,\delta}(\beta , y)$ up to a nonzero constant factor if $\Re \beta<0$ and $q_k\beta_k\notin \ZZ$, for $k=1,\dots,d$.
\end{theorem}
\begin{remark}
The argument of the exponential factor is not a polynomial but becomes a polynomial after a finite covering as in Remark \ref{finite_covering}. Therefore we can apply Corollary \ref{Hiensufficient2} to this covering. Performing backwards the change of variables in Section  \ref{section-change-coordinates}, our result gives a rapid decay cycle for $I_C(\beta;x)$.
\end{remark}
\begin{proof}
The proof starts with a preliminary reduction and then has three steps. First we build cycles depending on a parameter $\epsilon>0$ for which Corollary \ref{Hiensufficient} can be applied. We then show  that $F_{p,\delta}(\beta,y)$ is the limit when $\epsilon \, \too \, 0$ of the integrals over these cycles and finally we prove that these integrals are in fact independent of $\epsilon$.

If we perform the change of coordinates $t_k=e^{\sqrt{-1}\,\pi( 1 + 2 p_k + \delta_k)} r_k$ for $k=1,\ldots,d$, the image of the cycle $D_{p,\delta}$, defined in \eqref{Dpdelta} is just the positive quadrant $\RR^d_{>0}$, and we find

\begin{equation}\label{5}
F(\beta;y)=\int_{\RR^d_{>0}} e^{-\sqrt{-1}\,\pi\ideal{\nn+2p+\delta,\,\beta}}r^{-\beta-{\bf 1}} \exp{{\left(-\sum_{k=1}^d e^{\sqrt{-1}\,\pi\delta_k}r_k + \sum_{j=d+1}^n z_{j}r^{a(j)}\right)}} dr
\end{equation}
with $z_j=e^{\sqrt{-1}\,\pi\ideal{\nn+2p+\delta,\, a(j)}}y_{j}$, and $\Re z_n<0$.

We do not loose any information by replacing $z_j$ by $y_j$, and for the sake of simplicity we skip the constant $e^{-\sqrt{-1}\,\pi\ideal{\nn+2p+\delta,\,\beta}}$ and consider only the case $p=\delta=0$, hence reduce to the integral:
\begin{equation}\label{6}
F(\beta;y)=\int_{\RR^d_{>0}} r^{-\beta-{\bf 1}} \exp{{\left(-r_1-\cdots-r_d+ \sum_{j=d+1}^n y_{j}r^{a(j)}\right)}}dr.
\end{equation}

We remember that, by Lemma \ref{sufficientplus}, this integral is convergent when $\Re \beta_k <0$ for all $k=1,\ldots, d$ and $\Re (y_n r^{a(n)}) <0$.
Finally we are looking for cycles $\varUpsilon(\epsilon)$ such that the integral
\[H_{\varUpsilon(\epsilon)}(\beta;y):= \int_{\varUpsilon(\epsilon)} u^{-\beta-{\bf 1}} \exp{{\left(-u_1-\cdots-u_d+ \sum_{j=d+1}^n y_{j}u^{a(j)}\right)}}du
\]
tends to $F(\beta;y)$ when $\epsilon\,\too\, 0$. We denote $u$ the variable in this integral, instead of $r$, because we dedicate this last letter to a range included in $\RR$. The cycle $\widetilde{D}_{0,0}$ in the statement of Theorem \ref{limit-rapid-decay}, is the image of $\varUpsilon(\epsilon)$ by $t_k=-u_k$ for some $\epsilon>0$.

Inspecting the proof, the case of a general value of $p, \delta$ is a straightforward adaptation.

Recall that by Assumption \ref{assumption-2} we have that $a(j)\in \QQ^d_{>0}$ for all $j=d+1,\ldots, n$, $|a(j)|<1$, for $j=d+1,\ldots, n-1$ and also $|a|=a_1+\cdots + a_d>1$, setting $a:=a(n)=~^t(a_1,\ldots,a_d)$.

Let us first describe a product of cycles $C:= \gamma_1\times\cdots\times \gamma_d$ on the universal covering $(\wt{\CC^*})^d$. We consider the finite covering $\displaystyle(\CC^*)^d_v \too (\CC^*)^d_u$ of multidegree $(q_1,\dots,q_d)$, given by the formulas $v_k^{q_k}=u_k$, between two samples of the torus $(\CC^*)^d$. In Figure 2 we draw the projection of $\gamma_k$ on $\CC^*_{v_k}$, and for the projection on the space  $\CC^*_{u_k}$ we turn $q_k$ times on the circle of radius $\epsilon$ in the $k$--th component.

In Figure 2 the radius is $\epsilon^{\frac1{q_k}}$ on the $k$--th component. We choose the argument to be $0$ or $2q_k{\pi}$, on the two half--lines of Figure 2. The integrand is the same up to a constant factor on the $2^d$ different products of the $d$ half--lines in $(\wt{\CC^*})^d$. With these choices we can think of $C$ indifferently as a cycle on $(\wt{\CC^*})^d$, or as a twisted cycle on either $(\CC^*)^d_u$ or $(\CC^*)^d_v$.

\centerline{$\stackrel{\includegraphics[scale=0.50]{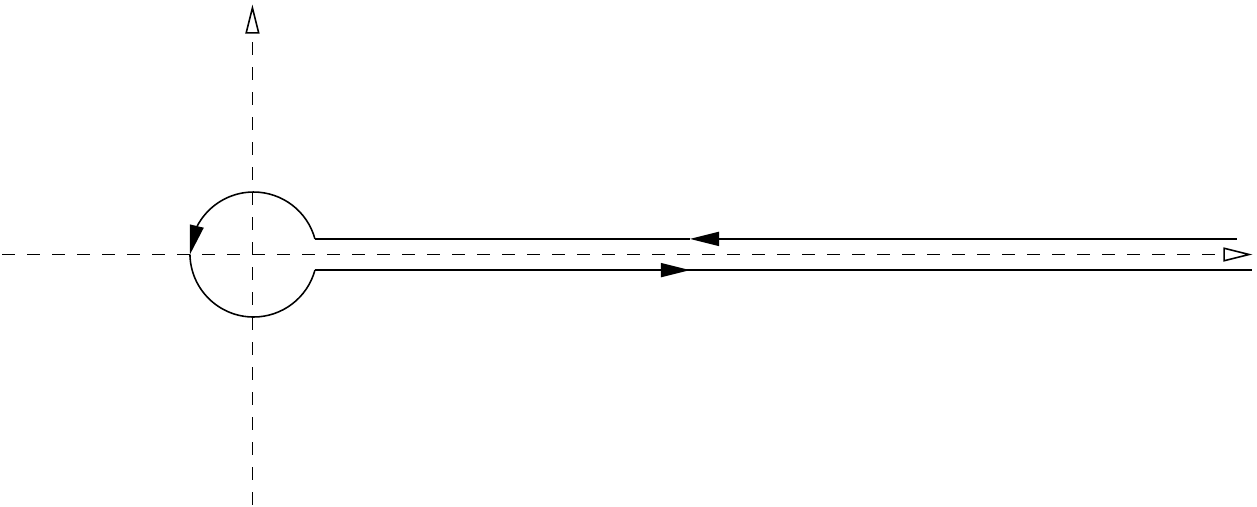}}{{\rm Figure\; 2}}$}

However there is a problem of convergence for the integral $H_C(\beta;y)$. The cycle $C$ is a union of products of the type $\left(S_\epsilon\right)^\eta\times \left([\epsilon,+\infty)\right)^\tau$. Here $S_\epsilon$ is a circle of radius $\epsilon>0$ and $\eta\cup \tau=\{1,\ldots,d\}$ is a partition of $\{1,\ldots,d\}$. On each piece with $\tau \ne \emptyset\ne \eta$ the integral is not convergent. Indeed, when $k\in \eta$ and $u_k$ varies in $S_\epsilon$, the argument of each monomial  $y_\ell  t^{a(\ell)}$ take all values mod $2\pi$. Therefore the monomial itself take arbitrarily large positive values, as well as $|t_\ell|$ for $\ell\in \tau$.

We shall build the cycle $\varUpsilon(\epsilon)$ as a deformed version of $C$. We identify $(\wt{\CC^*})^d$ with $\CC^d=\RR^d +\sqrt{-1} \,\RR^d$ and the covering map $(\wt{\CC^*})^d\to (\CC^*)^d$ with the map $\log r +\sqrt{-1}\,\theta\too (r_1 e^{\sqrt{-1}\,\theta_1},\dots, r_d e^{\sqrt{-1}\,\theta_d})$. We consider $(\wt{\CC^*})^d$ as fibered over $\RR_{>0}^d$, by the map $\log r + \sqrt{-1} \theta \too (r_1,\ldots,r_d)$, with fiber isomorphic to $\sqrt{-1}\, \RR^d$. The image of the restriction of this map to $\gamma_1\times\ldots\times \gamma_d$, is $[\epsilon,+\infty)^d$, with semi--algebraic fibers. The fiber over $(r_1,\ldots,r_d)$ is $\log r + \sqrt{-1}\, F_r$, where $F_r\subset\RR^d$ is a subset of arguments $\arg u:=(\arg u_1,\ldots,\arg u_d)$, which depends on $r$ in the following way:

\begin{enumerate}
\item Above each point $(r_1,\ldots,r_d)$ in the open quadrant $]\epsilon,+\infty)^d$, there are $2^d$ points with $\arg u \in F_r=\prod_{k=1}^ d\{0,2q_k\pi\}.$
\item Above the point $\{(\epsilon,\ldots,\epsilon)\}$, the argument $\arg u$ is in $F_r=\prod_{k=1}^ d[0,2q_k\pi]$.
\item In general, above the product $\{\epsilon\}^\eta\times \left(]\epsilon,+\infty)\right)^\tau$, the fiber has dimension $|\eta|$, the cardinality of $\eta$, with $2^{|\tau|}$ connected components. It is described in the universal covering $(\wt{\CC^*})^d$, by
\[
\begin{array}{ccc}
u_k\in &\log \epsilon + [0,2q_k\pi]\sqrt{-1} & \,\text{ if } k\in \eta,\\
u_\ell\in &\log r_\ell + \{0,2q_\ell\pi\} \sqrt{-1} & \, \text{ if } \ell\in \tau.
\end{array}
\]

\end{enumerate}

We choose instead of $\gamma_1\times\ldots\times \gamma_d$ a cycle $\varUpsilon(\epsilon)$
fibered over the subset of $\RR_{>0}^d$ described by the equation $r^a\geq \epsilon^{|a|}$,
which is the union of $2^d$ semi--algebraic strata:
\begin{enumerate}
\item  $\mathcal{S}_\emptyset=\{r\in\RR_{>0}^d \mid r^a>\epsilon^{|a|}\}$
\item $\mathcal{S}_\eta=\{r\in\RR_{>0}^d\mid r^a= \epsilon^{|a|},\,\text{ and } \eta=\{p\mid r_p=\min_{1\leq k\leq d} r_k\}\}$, if $\eta\ne\emptyset$.
\end{enumerate}
We shall sometimes write $re^{\sqrt{-1}\,\theta}$ instead of $\log r+\sqrt{-1}\,\theta$, since the abuse of notation fits better with the expression of the integral and it is clear from the context that when the target space is $\wt{\CC}$, arguments $\theta$ are to be considered in $\RR$.
\begin{definition} Description of the cycle $\varUpsilon(\epsilon)$:
\begin{enumerate}
\item The fiber of the support of $\varUpsilon(\epsilon)$ over the point $r\in \mathcal{S}_\eta$ is:
\[
\varUpsilon(\epsilon)_r:=\prod_{k\in \eta }(\log r_k+[0,2q_k\pi]\sqrt{-1})\times \prod_{\ell\in \tau}(\log r_\ell+\{0,2q_\ell\pi\}\sqrt{-1}).
\]
\item Let us take some $r\in S_\eta$. We have $r=(\rho^\eta;(r_k)_{k\in\tau})$, with $\rho^\eta\in \RR^\eta$ the point with all coordinates equal to $\rho$. When $\eta\ne\emptyset$, these data are subject to the conditions:
\begin{equation}\label{real-projection} \rho<\min((r_k)_{k\in\tau}), \quad r^a= \rho^{\sum_{j\in\eta}a_j}\prod_{k\in\tau}r_k^{a_k}=\epsilon^{|a|}.
\end{equation}
\item The projection of $S_\eta$ on the space $\RR_{>0}^{\tau}$ is a bijection $S_\eta\too U_{\eta}$ to the open subset  described by the inequalities :
\[
\epsilon^{|a|}<r_k^{\sum_{j\in\eta}a_j}\prod_{\ell\in\tau}r_\ell^{a_\ell} \text{ for any} \, k\in \tau.
\]
and the value of the $\eta$-coordinate $\rho$ of a point $r\in S_\eta$ is a function $\rho((r_k)_{k\in\tau})$ by implicit equation \eqref{real-projection}.
\item  Let $s\in\{0,\dots,d\}$ be the number of elements in $ \eta$. Then $\varUpsilon(\epsilon)_\eta$ is the union of  $2^{d-s}$ pieces. A typical piece is indexed by some $(\xi_k)_{k\in\tau}\in (\{0,1\})^\tau,$ and  parametrized by $\prod_{k\in\eta}[0,2q_k\pi]\times U_\eta\subset \prod_{k\in\eta}[0,2q_k\pi]\times \RR_{>0}^{\tau},$ in the following way:
\begin{equation}\label{param}
\xymatrix{\left((\theta_j)_{j\in\eta};(r_k)_{k\in\tau}\right)\ar[r]&\left((\rho e^{\sqrt{-1}\,\theta_j})_{j\in\eta};(r_ke^{2\sqrt{-1}\,\pi\xi_k q_k})_{k\in\tau}\right)}.
\end{equation}
\item We choose the coherent system of orientations inspired by the product of cycles $\gamma_k$, with the  circles positively oriented: we orient $\prod_{k\in\eta}[0,2q_k\pi]\times \RR_{>0}^{\tau}$, by its canonical orientation multiplied by the signature of the permutation $(\eta,\tau)$ of $\{1,\dots,d\}$, and by $(-1)^{d-\sum\xi_k}$.
\end{enumerate}
\end{definition}
In fact one can easily check that there is a radial isotopy from $\gamma_1\times\ldots\times \gamma_d$ to $\varUpsilon(\epsilon)$, which yields an oriented stratified isomorphism.

Indeed, for $r:=(r_1,\dots,r_d)\in(\RR_{>0}^d)$, with $r_k\geq \epsilon$ for all $k$, define $r_0=\min\{r_k\}$. On the half--line $\RR_{>0}r$ there is a unique
point $r'=(r_1',\cdots,r_d')$ with $\min\{r_k'\}=\epsilon$, and a unique point $\rho:=(\rho_1,\cdots,\rho_d)$, such that $\rho^a=\epsilon^{|a|}$. Let us consider $\rho_0=\min\{\rho_k\}$ and the linear multiplication on $\RR_{>0}r$ by the ratio $\rho_0/\epsilon=r_0 r^{-a/|a|}$, which depends continuously on $r$. Then the map
$\log r + i \theta \mapsto \log ((\rho_0/\epsilon) r)+\sqrt{-1}\,\theta$ from $\gamma_1\times\ldots\times \gamma_d$ to $\varUpsilon(\epsilon)$ is the mentioned radial isotopy.

Now we are proving that $\varUpsilon(\epsilon)$ is a rapid decay cycle and that the integrals $H_{\varUpsilon(\epsilon)}(\beta;y)$ are convergent. We work separately on each piece of the cycle.
\begin{remark}\label{Hiencycle} We work with the ramified version $(\CC^*)_v^d$ of the space $(\CC^*)_u^d$, that we introduced in the description of the cycles $C$ and $\varUpsilon(\epsilon)$. Our Corollary \ref{Hiensufficient2} is applied to $\varUpsilon(\epsilon)$ seen as a twisted cycle on this ramified  space, endowed with the pullback of the local system $\CC\cdot u^{-\beta-\nn}e^{M(u,y)}$. However all our calculations are done with the variable $u$. Indeed both variables $u$ and $v$ are equivalent for the control of any behaviour at infinity, since $\Vert u\Vert_1:=\sum |u_k|=\sum |v_k|^{q_k}$. Remark \ref{finite_covering} shows that when we come back to the original variable $t$, there is a smooth and finite covering map between the torii  $(\CC^*)_v^d$ and  $(\CC^*)_t^d$.
\end{remark}

Let us denote  $\varUpsilon(\epsilon)_\eta$ the union of the pieces of the cycle $\varUpsilon(\epsilon)$ above the stratum $\mathcal{S}_\eta$. By Assumption \ref{assumption} and the fact that $\Re y_nu^{a(n)}<0$ along $\varUpsilon(\epsilon)_\emptyset$, we can prove exactly as in Lemma \ref{sufficientplus} an inequality of type \eqref{dominant-exponent} for $u\in \varUpsilon(\epsilon)_\emptyset$.

Let us consider  a stratum with $\eta\ne\emptyset$. The last monomial of the argument of the exponential in
$F(\beta;y)$ satisfies:
\begin{equation}\label{last-monomial}
|y_nu^{a(n)}|=|y_n|\epsilon^{|a(n)|}.\end{equation}

For $\eta=\{1,\ldots,d\}$ the fiber over $S_{\eta}$ is a compact subset of $(\widetilde{\CC^*})^d$ and the integrand of $H_{\varUpsilon(\epsilon)}(\beta;y)$ is holomorphic over it, so there is nothing to prove. Let us assume for simplicity that $\eta=\{1,\ldots , s\}$ with $1\leq s< d$.
On the stratum $\mathcal{S}_{\eta}$ we have $r_1=\cdots =r_s=\rho <\epsilon$. Thus, if we imitate the proof of Lemma \ref{sufficientplus} (recall that
$\tau_A =\sigma\cup \{n\}$ in our case) we get the following upper bound for $\Re M(t,y)$ instead of inequality (\ref{Re-exponent}):
\begin{equation}\label{boundexp}
\Re M(u,y)\leq  \epsilon s - r_{s+1}-\ldots -r_d + |y_n| \epsilon^{|a|}+ K  (\epsilon s + \sum_{k=s+1}^d r_k)^{\kappa}
\end{equation}
Since $r_{s+1}+\cdots+r_d \leq \Vert u\Vert_1 =\rho s + r_{s+1}+\cdots+r_d \leq \epsilon s + r_{s+1}+\cdots+r_d$ for $u\in\varUpsilon(\epsilon)_\eta$, we still get an inequality of type (\ref{dominant-exponent}). There are constants $C_\eta, c_\eta>0$ (depending also on $y$) such that $\Re M (y,u)\leq C_\eta-c_\eta \Vert u \Vert_1$ for all $u\in \varUpsilon(\epsilon)_\eta$.

Since $\vert u^{-\beta-\nn}\vert\leq \Vert u\Vert_1^{{-|\Re\beta|-d}}$, the convergence of the integral  $H_{\varUpsilon(\epsilon)}(\beta;y)$, follows from the fact that on the part
$\cS_\eta$ of the cycle, the function under the integral is dominated by:
\[
\Vert u\Vert_1^{-|\Re\beta|-d}e^{C_\eta-c_\eta \Vert u \Vert_1}
\]
and for any $R>0$, there is a compact $K_R\subset \varUpsilon(\epsilon)$ such that for $u\in \varUpsilon(\epsilon)\setminus K_R$, one has $\Vert u \Vert_1>R$.

A closer look at the argument which proves \eqref{boundexp} shows that we can write the following upper bound for $|\Im M(u,y)|$:
\[
|\Im M(u,y)|\leq d\epsilon+|y_n|\epsilon^{|a(n)|}+K (d\epsilon +\Vert u\Vert_1)^\kappa.
\]
This upper bound, the relation \eqref{boundexp} and the fact that $0<\kappa<1$ prove that $\Im M(u,y)/\Re M (u,y)$ tends to zero as $\Vert u \Vert_1$ tends to infinity. Thus, for any $\delta \in]0,\frac{\pi}2[$ and outside a compact set  $K_\delta$, any $u\in \varUpsilon(\epsilon)$ satisfies
\begin{equation}\label{rapid-decay-arg}
\arg M(u,y)\in ]\pi-\delta,\pi+\delta[.
\end{equation}
In particular the argument of $M(u,y)$ tends to $\pi$.
Therefore if we use a compactification $X$ of $(\CC^*)^d$, a real blow--up $\pi : \wt{X}\too X$ of $X$ along $D$, and apply Corollary \ref{Hiensufficient2}, we obtain that $\varUpsilon(\epsilon)$ is a rapid decay cycle.

Let us prove that when $\Re \beta_k<0$ for all $k$ the integral $H_{\varUpsilon(\epsilon)}(\beta;y)$ tends, when $\epsilon\to 0$, to the integral \eqref{6} multiplied by the obvious factor $$\sum_{\xi \in\{0,1\}^d}(-1)^{d-|\xi|}\exp( 2\sqrt{-1}\,\pi\sum\beta_kq_k\xi_k)=\prod_{k=1}^d (\exp(2\sqrt{-1}\,\pi q_k\beta_k)-1).$$  Since \eqref{6} is clearly the limit of the piece of the integral $H_{\varUpsilon(\epsilon)}(\beta;y)$ over $\varUpsilon(\epsilon)_\emptyset$,
it suffices to show that the integrals over $\mathcal{S}_\eta$ for $\eta\ne \emptyset$ tend to zero.
Let us assume again for simplicity that $\eta=\{1,\ldots , s\}$ with $1\leq s\leq d$.
On each piece of $\varUpsilon(\epsilon)_\eta$ the parameters are
$$(\theta_1,\ldots,\theta_s,r_{s+1},\ldots,r_d)\in \prod_{k\in\eta}[0,2q_k\pi]\times U_{\eta}.$$
and the change of variables from the parametrization \eqref{param} induces in the different factors of the integrand the following results:
\[
\bigwedge_{k=1}^d\frac {du_k}{u_k}=(\sqrt{-1}\,d\theta_1)\wedge \cdots \wedge (\sqrt{-1}\,d\theta_s)\wedge\frac {dr_{s+1}}{r_{s+1}}\wedge \cdots \wedge\frac {dr_d}{r_d},
\]
\[
u^{-\beta}= \rho^{-\beta_1-\cdots-\beta_s}r_{s+1}^{-\beta_{s+1}}\cdots r_d^{-\beta_d} \exp\left(\sqrt{-1}\,\left(-\sum_{j=1}^s\beta_j\theta_j-\sum_{k=s+1}^d 2\pi\beta_kq_k\xi_k\right)\right),
\]
\begin{align*}
|u^{-\beta}|=&\rho^{-\Re(\beta_1+\cdots+\beta_s)}\prod_{\ell=s+1}^dr_\ell^{-\Re\beta_\ell}\exp\left(\sum_{j=1}^s\Im\beta_j\theta_j+\sum_{k=s+1}^d 2\pi\Im\beta_kq_k\xi_k\right)\\
\leq & \,\epsilon^{-\Re(\beta_1+\cdots+\beta_s)}\prod_{\ell=s+1}^dr_\ell^{-\Re\beta_\ell}\exp\left(\sum_{k=1}^d 2\pi|\Im\beta_k|q_k\right).
\end{align*}
From these inequalities and the fact that the real part of the exponent is bounded from above by
\[
C_\eta-c_\eta(r_{s+1}+\dots+r_d)
\] with $C_\eta,c_\eta \in \RR_{>0}$ independent of $\epsilon$, for $\epsilon\in ]0,\epsilon_0]$, we see that the integral over $\varUpsilon(\epsilon)_\eta$ tends to zero when $\epsilon\to 0$ as expected, because $-\Re(\beta_1+\cdots+\beta_s)>0$.

Finally let us prove that the integral $H_{\varUpsilon(\epsilon)}$ does not depend on $\epsilon$: Take $0<\epsilon_1<\epsilon_2$. We consider $\varUpsilon([\epsilon_1,\epsilon_2])$, the non compact $(d+1)$-cycle $$\bigcup_{\epsilon\in [\epsilon_1,\epsilon_2]}\{\epsilon\}\times\varUpsilon(\epsilon)$$
with oriented boundary $\{\epsilon_1\}\times\varUpsilon(\epsilon_1)-\{\epsilon_2\}\times\varUpsilon(\epsilon_2)$. Consider then for $R>\epsilon_2$ the compact cycle $\varUpsilon_R=\varUpsilon([\epsilon_1,\epsilon_2])\cap ([\epsilon_1,\epsilon_2] \times P_R)$ where $P_R$ is the polydisk
\[
P_R=\{u\in \CC^d \mid \vert u_1\vert\leq R, \dots, \vert u_d\vert \leq R\}.
\]
Integrals $H_{\varUpsilon(\epsilon)}$ are of the form $H_{\varUpsilon(\epsilon)}=\int_{\varUpsilon(\epsilon)}\omega$, where $\omega$ is a holomorphic form of degree $d$ independent of $\epsilon$ and hence it is a closed form.
We have
\[
0=\int_{\varUpsilon_R}d\omega=\int_{\p\varUpsilon_R}\omega.
\]
The boundary $\p\varUpsilon_R$ is equal to $$(\{\epsilon_1\}\times\varUpsilon(\epsilon_1))\cap ([\epsilon_1,\epsilon_2]\times P_R)-(\{\epsilon_2\}\times\varUpsilon(\epsilon_2))\cap ([\epsilon_1,\epsilon_2]\times P_R)+\p_R. $$ Since by examining the parametrization \eqref{param} we see that each $d$-dimensional piece of $\p_R$ is included in an hyperplane $u_j=R$, hence the restriction to it of $\omega$ is zero. We deduce that the integral of $\omega$ on $\varUpsilon(\epsilon_j)\cap P_R$ (which can replace $(\{\epsilon_j\}\times \varUpsilon(\epsilon_j))\cap([\epsilon_1,\epsilon_2]\times P_R)$ because $\omega$ does not depend on $\epsilon$) for $j=1,2$ are equal. Taking the limit when $R\too \infty$ we obtain the result
$$H_{\varUpsilon(\epsilon_1)}=H_{\varUpsilon(\epsilon_2).}$$

In the case of general $p,\delta$, we keep the same cycle and work with the integral
\[
H_{\varUpsilon(\epsilon)}(\beta;y):= \int_{\varUpsilon(\epsilon)} u^{-\beta-{\bf 1}} \exp{{\left(-\sum_{k=1}^d e^{\sqrt{-1}\,\pi\delta_k}u_k + \sum_{j=d+1}^n z_{j}u^{a(j)}\right)}}du
\] where $z_j=e^{\sqrt{-1}\,\pi\ideal{\nn+2p+\delta,\, a(j)}}y_{j}$ and the proof is essentially the same with only an easy modification of inequality \eqref{boundexp}.

In particular, the cycle $\widetilde{D}_{p,\delta}$ in the statement of Theorem \ref{limit-rapid-decay}, is the image of $\varUpsilon(\epsilon)$ by $t_k=u_k \cdot \exp(\sqrt{-1}\,\pi\langle{\bf 1}+2p+\delta,\, a(k)\rangle)$.
\end{proof}

\noindent {\bf Conclusion:} The integral $H_{\varUpsilon(\epsilon)}(\beta; y)$ is analytic as a function of
$\beta \in \CC^d$.

Reintroducing the constant $e^{-\sqrt{-1}\,\pi\ideal{\nn+2p+\delta,\,\beta}}$ we see that $e^{-\sqrt{-1}\,\pi\ideal{\nn+2p+\delta,\,\beta}}H_{\varUpsilon(\epsilon)}(\beta; y)$ is equal, when $\Re \beta_k<0$ for all $k$, to $$\prod_{k=1}^d (\exp{(2\sqrt{-1}\,\pi q_k\beta_k)}-1)F_{p,\delta}(\beta;y)$$ hence to its meromorphic continuation $$\prod_{k=1}^d (\exp{(2\sqrt{-1}\,\pi q_k\beta_k)}-1)\wt{F}_{p,\delta}(\beta;y)$$ outside the union of hyperplanes $\mathcal P$ described in Lemma \ref{meromorphic-continuation}.

When $q_k\beta_k\notin \ZZ$ for all $k\in \{1,\dots,d\}$, the factor $\prod_{k=1}^d (\exp{(2\sqrt{-1}\,\pi q_k\beta_k)}-1)$ is non zero and we obtain Gevrey series expansion for the integral along rapid decay cycles $H_{\varUpsilon(\epsilon)}(\beta;y)$.
To check this last claim we have to remark that the set of poles of the analytic continuation $\wt{F}_{p,\delta}(\beta; y)$ is contained in $\cP$ which is itself contained in the set defined by $\prod_{k=1}^d (\exp{(2\sqrt{-1}\,\pi q_k \beta_k)}-1)= 0$. This latter set is, under Assumption \ref{assumption}, the set of parameters $\beta=B_{\sigma}^{-1}\gamma$ such  that $\gamma$ is called \emph{resonant} for $B$ (see \cite[2.9]{GKZ90}).

Coming back to the general situation of Theorem \ref{Asymptotic-expansions-theorem}, the result of this theorem and the above considerations prove the following theorem :

\begin{theorem}\label{conclusion-theorem}
If Assumption \ref{assumption} is satisfied and $\gamma \in \CC^d$ is non resonant for $B$, then all the
Gevrey solutions of $M_B(\gamma)$ along the hyperplane $x_n=0$ can be described as linear combinations of
a fixed set of asymptotic expansions of integral solutions of type $I_C(\gamma,x)$ along rapid decay cycles.
\end{theorem}

\end{document}